\documentclass[a4paper,11pt]{article}
\usepackage{amsfonts,amsmath,amsthm,amssymb,hyperref,tikz}
\usepackage{geometry}
\usepackage{thmtools}
\usepackage{array}
\usepackage[utf8]{inputenc}
\usepackage[all,2cell]{xy}
\usepackage{bbm}
\usepackage{marginnote}
\usetikzlibrary{calc,intersections,arrows,decorations.markings}

\newcommand{\Id}{\mathbf{1}}
\newcommand{\ol}{\overline}

\title{}
\date{}
\author{M. Gualtieri\thanks{University of Toronto} \and M. Matviichuk\footnotemark[1] \and G. Scott\footnotemark[1]}

\newtheorem{theorem}{Theorem}[section]
\newtheorem*{theorem*}{Theorem}
\newtheorem{lemma}[theorem]{Lemma}
\newtheorem{corollary}[theorem]{Corollary}
\theoremstyle{definition}
\newtheorem{remark}[theorem]{Remark}
\newtheorem{example}[theorem]{Example}
\newcommand{\del}{{\partial}}
\newcommand{\ZZ}{\mathbb{Z}}
\newcommand{\eps}{\varepsilon}
\newcommand{\delbar}{\overline{\partial}}
\newcommand{\wt}{\widetilde}
\newcommand{\TT}{\mathbb{T}}

\newcommand\restr[2]{{
  \left.\kern-\nulldelimiterspace 
  #1 
  \vphantom{\big|} 
  \right|_{#2} 
  }}

\makeatletter
\newcommand{\tpitchfork}{%
  \vbox{
    \baselineskip\z@skip
    \lineskip-.52ex
    \lineskiplimit\maxdimen
    \m@th
    \ialign{##\crcr\hidewidth\smash{$-$}\hidewidth\crcr$\pitchfork$\crcr}
  }%
}
\makeatother

\makeatletter
\renewcommand\tableofcontents{%
    \@starttoc{toc}%
}
\makeatother
	
\theoremstyle{definition}
\newtheorem{definition}[theorem]{Definition}

\title{Deformation of Dirac structures via $L_\infty$ algebras}
\begin{document}
\renewcommand{\abstractname}{\vspace{-\baselineskip}}
\maketitle

\begin{abstract}
The deformation theory of a Dirac structure is controlled by a differential graded Lie algebra which depends on the choice of an auxiliary transversal Dirac structure; if the transversal is not involutive, one obtains an $L_\infty$ algebra instead.  We develop a simplified method for describing this $L_\infty$ algebra and use it to prove that the $L_\infty$ algebras corresponding to different transversals are canonically $L_\infty$--isomorphic.  In some cases, this isomorphism provides a formality map, as we show in several examples including (quasi)-Poisson geometry, Dirac structures on Lie groups, and Lie bialgebras.  Finally, we apply our result to a classical problem in the deformation theory of complex manifolds: we provide explicit formulas for the Kodaira-Spencer deformation complex of a fixed small deformation of a complex manifold, in terms of the deformation complex of the original manifold.
\end{abstract}

\section{Introduction}
Deformation theory is ordinarily treated using the language of differential graded Lie algebras (DGLA, for short), 
with small deformations given by the solutions to the Maurer-Cartan equation. 
However, for certain deformation problems it is necessary to use $L_\infty$ algebras instead, in which the Lie bracket on the differential complex is required to satisfy the Lie axioms only in cohomology, and its failure to be Lie on cochains is measured by a sequence of brackets of higher arity.  One of the main advantages of the $L_\infty$ formalism is that DGLAs may be equivalent as $L_\infty$ algebras without being isomorphic as DGLAs. Nevertheless, equivalent $L_\infty$ algebras give rise to equivalent moduli problems.  This leads to striking ``formality'' results, where one observes that a nontrivial DGLA is formal, in the sense of being $L_\infty$-equivalent to a DGLA with vanishing bracket, leading to the unobstructedness of the original moduli problem.

A notable example is Tian-Todorov's proof of unobstructedness of deformations for Calabi-Yau manifolds \cite{Ti87,Tod89}.
While the original proof proceeds by directly solving the Maurer-Cartan equation,
it was noticed later that the theorem follows from the $L_\infty$-formality of the  Kodaira-Spencer DGLA \cite{GoMi90}.
Another example is Hitchin's proof of unobstructedness of deformations for holomorphic Poisson manifolds~\cite{Hi12}, which was later seen to follow from the $L_\infty$-formality of 
the Koszul DGLA on the de Rham forms of a Poisson manifold \cite{FM12}.

In this paper, we focus on the deformation theory of a Dirac structure $M$, which is an involutive Lagrangian subbundle of a Courant algebroid; Dirac structures encompass a great many types of geometric structure including symplectic, Poisson, complex and generalized complex geometry.  Deformations of Dirac structures were first studied by Liu, Weinstein, and Xu, who observed that upon choosing a Dirac structure $L$ transverse to $M$, one obtains a Lie bracket $[\cdot,\cdot]_L$ on the de Rham complex of $M$ which makes it into a DGLA 
$(\Omega_M[1], d_M, [\cdot,\cdot]_L)$.  A small deformation of $M$ is then given by a solution to the Maurer-Cartan equation~\cite[Theorem 6.1]{LWX}:
\begin{equation}\label{eq:MC_equation_from_LWX}
d_M\omega + \frac{1}{2}[\omega,\omega]_L=0.
\end{equation}
Our main result addresses the nontrivial dependence of the above DGLA structure on the choice of Dirac transversal $L$:
\begin{theorem*}
Let $M$ be a Dirac structure in an exact Courant algebroid, and let $L$ and $L'$ be  Dirac structures transverse to $M$. 
Then there exists a canonical $L_\infty$-isomorphism between the differential graded Lie algebras $(\Omega_M[1], d_M, [\cdot,\cdot]_L)$ and $(\Omega_{M}[1], d_{M}, [\cdot,\cdot]_{L'})$.
\end{theorem*}
The construction of the $L_\infty$-isomorphism is given explicitly in Theorem~\ref{Thm:derived_bracket_construction_preserves_gauge_action},
yielding the above result in Corollary~\ref{cor:independence_of_L_infty_from_the_choice_of_complement}.
We also apply the $L_\infty$-isomorphism to Maurer-Cartan elements and verify the convergence of the resulting series in Theorem~\ref{mccomp}.

In general, the choice of a transverse Dirac structure may not be possible; in this case we simply choose an almost Dirac structure, i.e. a Lagrangian subbundle $L$ which may not be involutive.  With this choice, it was observed in \cite{Roy1}, \cite{KeWa} and \cite{FrZa} that the deformations of $M$ are controlled not by a DGLA but by a cubic $L_\infty$ algebra $(\Omega_M[1], d_M, [\cdot,\cdot]_L, [\cdot,\cdot,\cdot]_L)$; the Maurer-Cartan equation \eqref{eq:MC_equation_from_LWX} then acquires an extra cubic term (cf. \cite[(5.2)]{Roy1}).  Our result is actually stated in this more general situation: there is a canonical $L_\infty$-isomorphism relating the cubic $L_\infty$ structures associated to any pair of transversal almost Dirac structures.  We also drop the integrability assumption on $M$ itself, and our results still hold for what are known as ``curved'' $L_\infty$ algebras, which have an additional bracket of arity $0$ (see Theorem \ref{quantized_derived_brackets_integrability}, which includes the curved case).

The $L_\infty$ algebra associated to the pair $(M,L)$ was constructed in \cite[Lemma 2.6]{FrZa} using Voronov's derived bracket construction \cite{Vor} and Roytenberg's interpretation of a Courant algebroid 
as a graded symplectic manifold equipped with a cubic Hamiltonian function~\cite{Roy}. 
For exact Courant algebroids, we give a significantly simpler description of the $L_\infty$ algebra on the de Rham complex of $M$.   
Our approach involves studying the Dirac structures in terms of their corresponding pure spinors for the Clifford algebra of the ambient Courant algebroid.  We use a version of 
the $BV_\infty$ formalism \cite{BaVo}, which we view as a quantization of the derived bracket construction mentioned above.  This new method for computing the $L_\infty$ structures in question allows us to easily construct the required $L_\infty$-isomorphisms in our main Theorem.

The paper is organized as follows. In Section \ref{section:preliminaries} we give the necessary definitions and the main motivating idea. 
In Section \ref{section:proof} we present our method of constructing the $L_\infty$ algebras mentioned above and prove the main result. 
In Section \ref{section:action} we describe explicitly how the solutions of the Maurer-Cartan equation \eqref{eq:MC_equation_from_LWX} 
transform under the established $L_\infty$-isomorphism. 
In Section \ref{section:examples} we give several applications of our results.

One intriguing application is to the classical theory of deformations of complex manifolds. In Section~\ref{KSTH}, we explain how to use our main result to express the Kodaira-Spencer DGLA of a deformed complex manifold in terms of the DGLA of the original complex manifold. In effect, we describe the Kodaira-Spencer deformation complex of a deformation of a complex manifold. \\

\noindent \textit{Acknowledgements:} We thank Domenico Fiorenza and Yael Fr\'egier for discussions on the topics related to the subject of this paper.  We also thank Florian Sch\"atz and Marco Zambon for discussions about their forthcoming work~\cite{FlZa} which applies similar ideas to those here to the problem of deformations of presymplectic structures.  This research was supported by the Institut Henri Poincar\'e, a NSERC Discovery Grant as well as a grant from the Fondation Math\'ematique Jacques Hadamard.

\section{Notation and Motivation}\label{section:preliminaries}
\subsection{Dirac structures and dg Lie algebras}\label{dgladir}
Let $X$ be a smooth manifold equipped with a closed real 3-form $H$.
\begin{definition}
The \textbf{generalized tangent bundle} of $(X,H)$ is the bundle
\[
\mathbb{T}X := TX \oplus T^*X \rightarrow X,
\]
endowed with the bilinear form $\langle V + \xi, W + \eta \rangle = \frac{1}{2}(\xi(W) + \eta(V))$ and the \textbf{Courant bracket}
\[
[\![V + \xi, W + \eta]\!]_H = [V, W] + \mathcal{L}_V\eta - \iota_Wd\xi + \iota_V\iota_WH.
\]
on its sheaf of sections. 
\end{definition}
\begin{definition}
A subspace of $\mathbb{T}_xX = T_xX \oplus T_x^*X$ is \textbf{Lagrangian} if it is isotropic with respect to the split-signature bilinear form on $\mathbb{T}X$ and has dimension $\textrm{dim}(X)$. An \textbf{almost Dirac structure} is a Lagrangian subbundle $L \subseteq \mathbb{T}X$. It is a \textbf{Dirac structure} if, additionally, $L$ is involutive with respect to the Courant bracket.
\end{definition}
Let $M$ be a Dirac structure. The restriction of the Courant bracket to $\Gamma(M)$ (which we denote by $[\cdot, \cdot]_M$), together with the projection $a: L \rightarrow TM$, gives $M$ the structure of a Lie algebroid \cite{Co}. 
The bracket $[\cdot, \cdot]_M$ extends to the Schouten bracket on the space $\mathcal{X}_M = \Gamma (\wedge M)$ of $M$-vector fields by letting $[f,s]=-\iota_{a(s)}df$, $f\in \mathcal{O}_X, s\in \Gamma(M)$ and requiring it to be a graded derivation in each entry. Also, the space $\Omega_M=\Gamma(\wedge M^*)$ of $M$-forms has the differential $d_M$ given by the formula
\begin{align*}
d_M\omega(V_0, \dots, V_p) &:= \sum_{i}(-1)^ia(V_i)\left( \omega(V_0, \dots, \hat{V}_i, \dots, V_p)\right)\\
& \ \ \ \  + \sum_{i < j} (-1)^{i+j} \omega([V_i, V_j]_M, V_0, \dots, \hat{V}_i, \dots, \hat{V}_j, \dots, V_p).
\end{align*}
We will soon see that choosing a Dirac structure $L$ transverse to $M$ produces the additional structure of a bracket on $\Omega_M$, forming a differential graded Lie algebra. We recall the definition. 



\begin{definition} A \textbf{differential graded Lie algebra} (DGLA) is a triple $(V, d, [\cdot, \cdot])$ consisting of a $\ZZ$-graded complex $(V, d)$, and a degree zero map $[\cdot, \cdot]: V \otimes V \rightarrow V$ called the \textbf{bracket} such that $d$ is a derivation of the bracket, the bracket is graded skew-symmetric, and the bracket satisfies the graded Jacobi identity. That is, for all homogeneous $v, w, z \in V$, the following conditions hold:
\begin{enumerate}
\item[i)] $d[v, w] = [dv, w] + (-1)^{|v|}[v, dw]$, 
\item[ii)] $[v, w] = -(-1)^{|v||w|}[w,v]$,
\item[iii)] $[v, [w, z]] + (-1)^{|v||w| + |v||z|}[w, [z, v]] + (-1)^{|w||z| + |v||z|}[z, [w, v]]=0$. 
\end{enumerate}
\end{definition}

\begin{example}[The DGLA associated to a transverse pair of Dirac structures]
Given a pair $(M,L)$ of transverse Dirac structures, the inner product on $\mathbb{T}X$ induces an isomorphism $M^* \cong L$. Under this identification, the Lie algebroid bracket $[\cdot , \cdot]_L$ induces a bracket (also written $[\cdot , \cdot ]_L$) on the sheaf of sections of $M^*$. It was shown in~\cite{LWX} that this bracket extends to a DGLA bracket on the de Rham complex of $M$ (shifted in degree so that $\Omega^1_M$ lies in degree zero), yielding the DGLA 
\begin{equation}
(\Omega_M[1], d_M, [\cdot, \cdot]_L).  
\end{equation}
\end{example}

\begin{definition}
A \textbf{Maurer-Cartan element} of a DGLA $(V, d, [\cdot, \cdot])$ is an element $a \in V$ of degree 1 which satisfies the \textbf{Maurer-Cartan equation}
\begin{equation}\label{eqn:maurercartan}
da + \frac{1}{2}[a, a] = 0.
\end{equation}
\end{definition}

\begin{example}[Maurer-Cartan elements as deformations of Dirac structures]
Given transverse Dirac structures $L$ and $M$, every $\varepsilon \in \Omega^2_L$ defines another Lagrangian subbundle $L_{\varepsilon} := \{x + \varepsilon(x) \mid x \in L\}$, also transverse to $M$. Conversely, every Lagrangian subbundle transverse to $M$ may be realized as $L_{\varepsilon}$ for some $\varepsilon \in \Omega^2_L$. This $L_{\varepsilon}$ is involutive -- and thereby defines a Dirac structure -- precisely if it is a Maurer-Cartan element (\cite[Theorem 6.1]{LWX}), i.e. if $d_L\varepsilon + \frac{1}{2}[\varepsilon,\varepsilon]_M=0$. In this way, small deformations of $L$ are in bijection with Maurer-Cartan elements of the DGLA $(\Omega_L[1], d_L, [ \cdot, \cdot]_M)$, where by ``small'' we mean those that remain transverse to $M$.\\
\end{example}

\subsection{Main motivating idea}

Our focus in this paper is the situation where we have \emph{two} different Dirac structures $L, L_{\varepsilon}$ transverse to $M$, and we wish to compare the corresponding DGLA structures $(\Omega_M[1], d_M, [\cdot , \cdot]_L)$ and $(\Omega_M[1], d_M, [\cdot ,\cdot ]_{L_{\varepsilon}})$.  These DGLAs, while not isomorphic, both describe deformations of $M$:  any almost Dirac structure $M'$ which is transverse to both $L$ and $L_{\varepsilon}$ may expressed as a graph, either of a map $\wt\omega: M \rightarrow L$, or alternatively of a map $\wt\omega_\eps: M \rightarrow L_{\eps}$, as shown in Figure~\ref{FigML}.  Identifying $L$ and $L_{\eps}$ with $M^*$ using the canonical inner product on $\mathbb{T}X$, each of $\wt\omega, \wt\omega_\eps$ may be viewed respectively as elements $\omega, \omega_\eps$ of $\Omega^2_M$. Therefore, we obtain a map $\Omega^2_M\to\Omega^2_M$, defined away from the non-transverse locus, taking $\omega$ to the element $\omega_\eps$ such that 
\[
\textrm{graph}(\wt\omega: M \rightarrow M^*\cong L) = \textrm{graph}(\wt\omega_{\eps}: M \rightarrow M^* \cong L_{\eps}).
\]
The equality of the above graphs may be phrased as the identity 
$\wt\omega_{\eps}(\mathbbm{1}_M-\eps \wt\omega) = (\mathbbm{1}_L+\eps) \wt\omega$; applying the inner product and using $\langle\wt\omega u, v\rangle =\omega(u,v)$ and $\langle\wt\omega_\eps u, v\rangle =\omega_\eps(u,v)$ for $u,v\in M$, we solve for $\omega_\eps$:
\begin{equation}\label{keymcmap}
\omega_{\eps} = \omega (\mathbbm{1} - \eps \omega)^{-1} = \omega + \omega\eps\omega + \omega\eps\omega\eps\omega + \cdots.
\end{equation}
This formula therefore sends Maurer-Cartan elements of $(\Omega_M[1], d_M, [\cdot , \cdot]_L)$ to Maurer-Cartan elements of $(\Omega_M[1], d_M, [\cdot ,\cdot ]_{L_{\varepsilon}})$. We also see that transversality of $M'$ to $L,L_\eps$ holds so long as $(\mathbbm{1}-\eps\omega)$ is invertible as an endomorphism of $M$.

\begin{figure}[!ht]
\centering\begin{tikzpicture}[scale=2,decoration={
    markings,
    mark=at position 0.5 with {\arrow{>}}}]
    \draw (0.5,0.5) coordinate (lp1) -- (2,2) coordinate[label=above:{$L_{\varepsilon}$}] (lp2);
		\draw (0,1) coordinate (m1) -- (2.5,1) coordinate[label=right:M] (m2) ;
    \coordinate (O) at (intersection of lp1--lp2 and m1--m2);
	\coordinate (c1) at (0.7,0.4);
	\coordinate[label=above:L] (c2) at ($(c1)!3!(O)$); 
	\draw (c1)--(c2);
	\coordinate (c3) at ($(c1)!2.2!(O)$); 
	\coordinate (f3) at ($(c3)+(1,0)$); 
	\coordinate (lp3) at (intersection of lp1--lp2 and c3--f3);
	\draw[->] (c3) to node [above] {$\varepsilon$} (lp3);
	\coordinate (mp1) at (0, 0.7); 
	\coordinate[label=right:{$M'$}] (mp2) at ($(mp1)!2.5!(O)$); 
	\draw[red] (mp1)--(mp2);
	\coordinate (mp3) at ($(mp1)!2.3!(O)$);
	\coordinate (g3) at ($(mp3)-0.3*(c2)+0.3*(c1)$);
	\coordinate (m4) at (intersection of m1--m2 and mp3--g3);
	\coordinate (h3) at ($(mp3)-0.3*(lp2)+0.3*(lp1)$); 
	\coordinate (m3) at (intersection of m1--m2 and mp3--h3); 
	\draw[postaction={decorate}] (m4) to node [right] {$\omega$} (mp3);
	\draw[postaction={decorate}] (m3)to node [left] {$\omega_\varepsilon$} (mp3);
\end{tikzpicture}
\caption{Two descriptions of a deformation $M'$ of $M$ using the complements $L, L_{\varepsilon}$ }\label{FigML}
\end{figure}
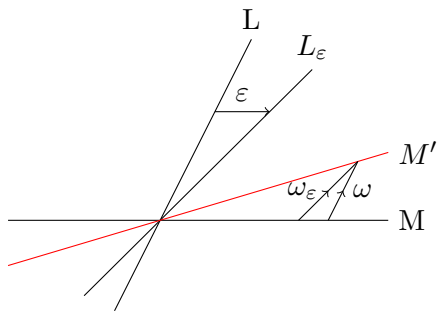

We conclude from the above that while the two DGLA structures are not necessarily isomorphic, there is a correspondence between their Maurer-Cartan elements.  In Section~\ref{section:proof}, we show that these DGLAs are actually isomorphic as $L_\infty$ algebras, and that the successive terms of Equation~\ref{keymcmap} are generated by the respective components of the $L_\infty$ morphism. We now briefly review the key aspects of $L_{\infty}$ algebras which we need for the construction.

\subsection{$L_\infty$ structures and morphisms}\label{linfsym}

Let $V = \oplus_{i \in \mathbb{Z}} V_i$ be a graded vector space. Given a permutation $\sigma \in S_n$ and elements $v_1, \dots, v_n \in V$, 
we define the \textbf{Koszul sign} $\eps(\sigma)$ to be the product of the factors $(-1)^{|v_i||v_{i+1}|}$ for each transposition of $v_i$ with $v_{i+1}$ 
involved in the permutation. Although $\eps(\sigma)$ depends on $|v_1|, \dots, |v_n|$, this dependence is suppressed in the notation. 

\begin{definition} 
The \textbf{$n^{th}$ graded exterior algebra} of $V$ is the graded algebra $\wedge V $ consisting of elements of $\oplus_{n\ge0}V^{\otimes n}$ 
fixed by the action of $S_n$
\[
v_1 \otimes \dots \otimes v_n \mapsto sign(\sigma)\eps(\sigma)v_{\sigma_1} \otimes \dots \otimes v_{\sigma_n}.
\]
while the \textbf{$n^{th}$ graded symmetric algebra} of $V$ is the graded algebra $S(V)$ consisting of elements of $\oplus_{n\ge0}V^{\otimes n}$ fixed by the action of $S_n$
\[
v_1 \otimes \dots \otimes v_n \mapsto \eps(\sigma)v_{\sigma_1} \otimes \dots \otimes v_{\sigma_n}.
\]
\end{definition}
Equivalently, $\wedge V$ and $S(V)$ are the quotients of $\otimes V$ by the ideal generated by elements of the form 
$v \otimes w + (-1)^{|v||w|}w \otimes v$ and the ideal generated by elements of the form $v \otimes w - (-1)^{|v||w|}w \otimes v$, respectively. 
The graded symmetric algebra is also endowed with a canonical coalgebra structure: 
the counit is given by the canonical projection $\eps_S: S(V) \rightarrow S^{0}V$, 
and the comultiplication is given on $v \in V \cong S^1(V)$ by 
\[
\Delta_S(v) = 1 \otimes v + v \otimes 1
\]
and extending it as a morphism of algebras $\Delta_S: S(V) \rightarrow S(V)\otimes S(V)$ . 


\begin{definition}\label{def:linf2}
An $L_{\infty}$-algebra is a pair $(V, Q)$ where $V$ is a graded vector space and $Q$ is a coderivation of degree $1$ on the coalgebra
${S}(V[1])$ which satisfies $Q^2=0$.
\end{definition}

One may equivalently define an $L_{\infty}$-algebra as a sequence of degree $1$ linear maps $m_n: S^n(V[1]) \rightarrow V[1]$ for $n \geq 0$ 
that satisfy the following 
\textbf{generalized Jacobi equations}, for $n \geq 0$:
\begin{equation}\label{eq:generalized_Jacobi}
\sum_{k=0}^n \sum_{\sigma \in \textrm{Sh}(k, n-k)}  \eps(\sigma) m_{n-k+1}(m_k(v_{\sigma_1}, \dots, v_{\sigma_k}), v_{\sigma_{k+1}}, \dots, v_{\sigma_{n}}) = 0
\end{equation}
where $\textrm{Sh}(i, n-i)$ consists of all permutations with $\sigma_1 < \dots < \sigma_i$ and $\sigma_{i+1} < \dots < \sigma_{n}$. 
The coderivation $Q$ can be uniquely recovered from the collection of maps $m_n$, $n\ge 0$, which may be viewed as the successive components of an expansion of $Q$ in Taylor series.
The conditions for $Q$ to be a coalgebra differential translate precisely into the conditions for these $m_n$ to satisfy 
the generalized Jacobi equations.  
If the $m_0$ vanishes, the $L_{\infty}$-structure is called \textbf{flat}, although this is often tacitly assumed when dealing with $L_{\infty}$ algebras.

Yet another equivalent way to encode an $L_\infty$ algebra, useful in practice, is to specify skew-symmetric maps $\ell_n$, rather than symmetric ones $m_n$.
Specifically, one can get degree $(2-n)$ maps $\ell_n:\wedge^n V \to V$, $n\ge0$ by letting
\begin{equation}\label{eq:symm_to_skew}
 \ell_n(v_1,\dots,v_n) = (-1)^n (-1)^{\sum_{i=1}^n(n-i)|v_i|}m_n(v_1,\dots,v_n).
\end{equation}
The maps $\ell_n$ satisfy a version of generalized Jacobi equations, which can be obtained from \eqref{eq:generalized_Jacobi} using 
the substitution \eqref{eq:symm_to_skew} (one gets more signs to keep track of, so in this respect the maps $m_n$ are more convenient).


\begin{definition}
A \textbf{(iso)morphism} between two $L_{\infty}$-algebras $(V, Q_V)$, $(W, Q_W)$ is a (iso)morphism of differential graded coalgebras
\[
F:({S}(V[1]), Q_V) \rightarrow ({S}(W[1]), Q_W)
\]
\end{definition}

Equivalently, an $L_\infty$ morphism is a collection of degree $0$ maps
$f_n: S^n(V[1]) \rightarrow W[1]$ satisfying the relations, for $n\ge1$, $v_i\in V[1]$:

\begin{multline}\label{eq:L_inf_relations}
\begin{aligned}
 \sum_{j\geq 1} \sum_{\substack{k_1+\dots+k_j=n\\ \sigma\in \textrm{Sh}(k_1,\dots,k_j)}} 
\tfrac{1}{j!}\eps(\sigma) 
m_j^W\big(f_{k_1}(v_{\sigma_1},\dots,v_{\sigma_{k_1}}), \dots, f_{k_j}(v_{\sigma_{k_1+\dots+k_{j-1}+1}}, \dots, v_{\sigma_n})\big)\\
=\sum_{k=0}^n \sum_{\sigma \in \textrm{Sh}(k, n-k)}  \eps(\sigma) f_{n-k+1}(m_k^V(v_{\sigma_1}, \dots, v_{\sigma_k}), v_{\sigma_{k+1}}, \dots, v_{\sigma_{n}}).
\end{aligned}
\end{multline}
For instance, assuming the $L_\infty$ algebras are flat, the relation for $n=1$ states that $m_1^W(f_1(v_1))=f_1(m_1^V(v_1))$, i.e., $f_1$ is a morphism of cochain complexes; for $n=2$ the statement is:
\begin{multline}\label{eq:L_inf_relation_2}
\begin{aligned}
m_2^W(f_1(v_1),f_1(v_2)) - f_1(m_2^V(v_1,v_2)) &= \\
- m_1^W(f_2(v_1,v_2)) + f_2(m_1^V(v_1),&v_2) + (-1)^{(|v_1|+1)(|v_2|+1)}f_2(m_1^V(v_2),v_1),
\end{aligned}
\end{multline}
implying that $f_1$ induces a Lie algebra homomorphism after passing to cohomology.

An $L_\infty$ morphism $\{f_n\}_{n\ge1}$ between two flat $L_\infty$ algebras $(V, Q_V)$, $(W, Q_W)$ is called a \textbf{quasi-isomorphism} if the map $f_1$ induces an isomorphism on cohomology $H_{m_1^V}(V)\to H_{m_1^W}(W)$. 
A flat $L_\infty$ algebra $(V, Q)$ is called \textbf{formal}, or $L_\infty$-formal, if it is quasi-isomorphic to the Lie algebra $(H_{\ell_1}(V), \ell_2)$.

Any DGLA $(V, d, [\cdot, \cdot])$ defines an $L_{\infty}$-structure with $\ell_1$ given by $d$, $\ell_2$ given by the bracket, and all other terms equal to zero. Any DGLA morphism induces a morphism of their corresponding $L_\infty$ algebras. Recall that our goal is to construct an $L_{\infty}$-isomorphism from $(\Omega_M[1], d_M, [\cdot , \cdot]_L)$ to $(\Omega_M[1], d_M, [\cdot ,\cdot ]_{L_{\varepsilon}})$. One technique for constructing $L_{\infty}$-isomorphisms is to exponentiate a coderivation. Given a graded vector space, a map $C: S^2(V) \rightarrow V$ extends to a coderivation 
\begin{align*}
S(V) &\rightarrow S(V)\\
v_1 \odot \dots \odot v_n &\mapsto \sum_{i < j} \eps(\sigma) C(v_i, v_j) \odot v_1 \odot \cdots \odot \widehat{v_i} \odot \cdots \odot \widehat{v_j} \odot \cdots \odot v_n
\end{align*}
where $\sigma$ is the permutation bringing $v_i, v_j$ to the front of $v_1, \dots, v_n$. Restricted to $S^5(V)$, for example, the term in the sum above corresponding to $i = 2, j = 4$ can be represented diagramatically as
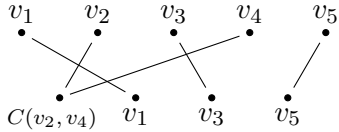
\begin{figure}[ht]
\centering
\begin{tikzpicture}

\node (a1) at (0,1) {};
\node (a2) at (1,1) {};
\node (a3) at (2,1) {};
\node (a4) at (3,1) {};
\node (a5) at (4,1) {};
\draw[fill= black] (a1) circle(0.4mm) node[above] {$v_1$};
\draw[fill= black] (a2) circle(0.4mm) node[above] {$v_2$};
\draw[fill= black] (a3) circle(0.4mm) node[above] {$v_3$};
\draw[fill= black] (a4) circle(0.4mm) node[above] {$v_4$};
\draw[fill= black] (a5) circle(0.4mm) node[above] {$v_5$};

\node[above] (b1) at (0.5,0) {};
\node[above] (b2) at (1.5,0) {};
\node[above] (b3) at (2.5,0) {};
\node[above] (b4) at (3.5,0) {};
\draw[fill= black] (b1) circle(0.4mm) node[below]  {\scriptsize $C(v_2, v_4)$ \ \ };
\draw[fill= black] (b2) circle(0.4mm) node[below] {$v_1$};
\draw[fill= black] (b3) circle(0.4mm) node[below] {$v_3$};
\draw[fill= black] (b4) circle(0.4mm) node[below] {$v_5$};

\draw (a1) -- (b2);
\draw (a2) -- (b1);
\draw (a3) -- (b3);
\draw (a4) -- (b1);
\draw (a5) -- (b4);
\end{tikzpicture}
\end{figure}

The sum of all terms contributing to $\restr{C}{S^5(V)}$ is taken over all ${5 \choose 2} = 10$ such forests, one for each choice of which two terms to compose and bring to the front, and each weighted by the Koszul sign. The coderivation $C$ may be composed with itself; its $n$-fold composition $C^n$ can be represented by concatenating the forest diagrams. For example, we can represent the $S^{4}(V) \rightarrow S^1(V)$ component of $C^3$ as a sum over the following 18 rooted trees, each weighted by Koszul signs.

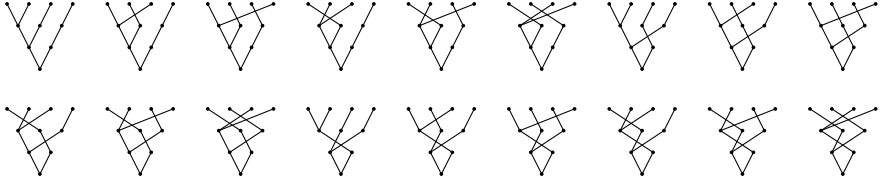
\begin{figure}[ht]
\centering
\begin{tikzpicture}[scale = 0.8]
\def\Horspacing{1.65}
\def\Graphscaling{0.36}
\def\Verspacing{-1.74}

\begin{scope}[shift={({0*\Horspacing}, 0)}, scale=\Graphscaling]

\node (a1) at (0,1) {};
\node (a2) at (1,1) {};
\node (a3) at (2,1) {};
\node (a4) at (3,1) {};
\node (b1) at (0.5,0) {};
\node (b2) at (1.5,0) {};
\node (b3) at (2.5,0) {};
\node (c1) at (1, -1) {};
\node (c2) at (2, -1) {};
\node (d1) at (1.5, -2) {};

\draw[fill= black] (a1) circle(0.6mm);
\draw[fill= black] (a2) circle(0.6mm);
\draw[fill= black] (a3) circle(0.6mm);
\draw[fill= black] (a4) circle(0.6mm);
\draw[fill= black] (b1) circle(0.6mm);
\draw[fill= black] (b2) circle(0.6mm);
\draw[fill= black] (b3) circle(0.6mm);
\draw[fill= black] (c1) circle(0.6mm);
\draw[fill= black] (c2) circle(0.6mm);
\draw[fill= black] (d1) circle(0.6mm);

\draw (a1.center) -- (b1.center);
\draw (a2.center) -- (b1.center);
\draw (a3.center) -- (b2.center);
\draw (a4.center) -- (b3.center);

\draw (b1.center) -- (c1.center);
\draw (b2.center) -- (c1.center);
\draw (b3.center) -- (c2.center);

\draw (c1.center) -- (d1.center);
\draw (c2.center) -- (d1.center);
\end{scope}

\begin{scope}[shift={({1*\Horspacing}, 0)}, scale=\Graphscaling]

\node (a1) at (0,1) {};
\node (a2) at (1,1) {};
\node (a3) at (2,1) {};
\node (a4) at (3,1) {};
\node (b1) at (0.5,0) {};
\node (b2) at (1.5,0) {};
\node (b3) at (2.5,0) {};
\node (c1) at (1, -1) {};
\node (c2) at (2, -1) {};
\node (d1) at (1.5, -2) {};

\draw[fill= black] (a1) circle(0.6mm);
\draw[fill= black] (a2) circle(0.6mm);
\draw[fill= black] (a3) circle(0.6mm);
\draw[fill= black] (a4) circle(0.6mm);
\draw[fill= black] (b1) circle(0.6mm);
\draw[fill= black] (b2) circle(0.6mm);
\draw[fill= black] (b3) circle(0.6mm);
\draw[fill= black] (c1) circle(0.6mm);
\draw[fill= black] (c2) circle(0.6mm);
\draw[fill= black] (d1) circle(0.6mm);

\draw (a1.center) -- (b1.center);
\draw (a2.center) -- (b2.center);
\draw (a3.center) -- (b1.center);
\draw (a4.center) -- (b3.center);

\draw (b1.center) -- (c1.center);
\draw (b2.center) -- (c1.center);
\draw (b3.center) -- (c2.center);

\draw (c1.center) -- (d1.center);
\draw (c2.center) -- (d1.center);
\end{scope}

\begin{scope}[shift={({2*\Horspacing}, 0)}, scale=\Graphscaling]

\node (a1) at (0,1) {};
\node (a2) at (1,1) {};
\node (a3) at (2,1) {};
\node (a4) at (3,1) {};
\node (b1) at (0.5,0) {};
\node (b2) at (1.5,0) {};
\node (b3) at (2.5,0) {};
\node (c1) at (1, -1) {};
\node (c2) at (2, -1) {};
\node (d1) at (1.5, -2) {};

\draw[fill= black] (a1) circle(0.6mm);
\draw[fill= black] (a2) circle(0.6mm);
\draw[fill= black] (a3) circle(0.6mm);
\draw[fill= black] (a4) circle(0.6mm);
\draw[fill= black] (b1) circle(0.6mm);
\draw[fill= black] (b2) circle(0.6mm);
\draw[fill= black] (b3) circle(0.6mm);
\draw[fill= black] (c1) circle(0.6mm);
\draw[fill= black] (c2) circle(0.6mm);
\draw[fill= black] (d1) circle(0.6mm);

\draw (a1.center) -- (b1.center);
\draw (a2.center) -- (b2.center);
\draw (a3.center) -- (b3.center);
\draw (a4.center) -- (b1.center);

\draw (b1.center) -- (c1.center);
\draw (b2.center) -- (c1.center);
\draw (b3.center) -- (c2.center);

\draw (c1.center) -- (d1.center);
\draw (c2.center) -- (d1.center);
\end{scope}

\begin{scope}[shift={({3*\Horspacing}, 0)}, scale=\Graphscaling]

\node (a1) at (0,1) {};
\node (a2) at (1,1) {};
\node (a3) at (2,1) {};
\node (a4) at (3,1) {};
\node (b1) at (0.5,0) {};
\node (b2) at (1.5,0) {};
\node (b3) at (2.5,0) {};
\node (c1) at (1, -1) {};
\node (c2) at (2, -1) {};
\node (d1) at (1.5, -2) {};

\draw[fill= black] (a1) circle(0.6mm);
\draw[fill= black] (a2) circle(0.6mm);
\draw[fill= black] (a3) circle(0.6mm);
\draw[fill= black] (a4) circle(0.6mm);
\draw[fill= black] (b1) circle(0.6mm);
\draw[fill= black] (b2) circle(0.6mm);
\draw[fill= black] (b3) circle(0.6mm);
\draw[fill= black] (c1) circle(0.6mm);
\draw[fill= black] (c2) circle(0.6mm);
\draw[fill= black] (d1) circle(0.6mm);

\draw (a1.center) -- (b2.center);
\draw (a2.center) -- (b1.center);
\draw (a3.center) -- (b1.center);
\draw (a4.center) -- (b3.center);

\draw (b1.center) -- (c1.center);
\draw (b2.center) -- (c1.center);
\draw (b3.center) -- (c2.center);

\draw (c1.center) -- (d1.center);
\draw (c2.center) -- (d1.center);
\end{scope}

\begin{scope}[shift={({4*\Horspacing}, 0)}, scale=\Graphscaling]

\node (a1) at (0,1) {};
\node (a2) at (1,1) {};
\node (a3) at (2,1) {};
\node (a4) at (3,1) {};
\node (b1) at (0.5,0) {};
\node (b2) at (1.5,0) {};
\node (b3) at (2.5,0) {};
\node (c1) at (1, -1) {};
\node (c2) at (2, -1) {};
\node (d1) at (1.5, -2) {};

\draw[fill= black] (a1) circle(0.6mm);
\draw[fill= black] (a2) circle(0.6mm);
\draw[fill= black] (a3) circle(0.6mm);
\draw[fill= black] (a4) circle(0.6mm);
\draw[fill= black] (b1) circle(0.6mm);
\draw[fill= black] (b2) circle(0.6mm);
\draw[fill= black] (b3) circle(0.6mm);
\draw[fill= black] (c1) circle(0.6mm);
\draw[fill= black] (c2) circle(0.6mm);
\draw[fill= black] (d1) circle(0.6mm);

\draw (a1.center) -- (b2.center);
\draw (a2.center) -- (b1.center);
\draw (a3.center) -- (b3.center);
\draw (a4.center) -- (b1.center);

\draw (b1.center) -- (c1.center);
\draw (b2.center) -- (c1.center);
\draw (b3.center) -- (c2.center);

\draw (c1.center) -- (d1.center);
\draw (c2.center) -- (d1.center);
\end{scope}

\begin{scope}[shift={({5*\Horspacing}, 0)}, scale=\Graphscaling]

\node (a1) at (0,1) {};
\node (a2) at (1,1) {};
\node (a3) at (2,1) {};
\node (a4) at (3,1) {};
\node (b1) at (0.5,0) {};
\node (b2) at (1.5,0) {};
\node (b3) at (2.5,0) {};
\node (c1) at (1, -1) {};
\node (c2) at (2, -1) {};
\node (d1) at (1.5, -2) {};

\draw[fill= black] (a1) circle(0.6mm);
\draw[fill= black] (a2) circle(0.6mm);
\draw[fill= black] (a3) circle(0.6mm);
\draw[fill= black] (a4) circle(0.6mm);
\draw[fill= black] (b1) circle(0.6mm);
\draw[fill= black] (b2) circle(0.6mm);
\draw[fill= black] (b3) circle(0.6mm);
\draw[fill= black] (c1) circle(0.6mm);
\draw[fill= black] (c2) circle(0.6mm);
\draw[fill= black] (d1) circle(0.6mm);

\draw (a1.center) -- (b2.center);
\draw (a2.center) -- (b3.center);
\draw (a3.center) -- (b1.center);
\draw (a4.center) -- (b1.center);

\draw (b1.center) -- (c1.center);
\draw (b2.center) -- (c1.center);
\draw (b3.center) -- (c2.center);

\draw (c1.center) -- (d1.center);
\draw (c2.center) -- (d1.center);
\end{scope}


\begin{scope}[shift={({6*\Horspacing}, 0)}, scale=\Graphscaling]

\node (a1) at (0,1) {};
\node (a2) at (1,1) {};
\node (a3) at (2,1) {};
\node (a4) at (3,1) {};
\node (b1) at (0.5,0) {};
\node (b2) at (1.5,0) {};
\node (b3) at (2.5,0) {};
\node (c1) at (1, -1) {};
\node (c2) at (2, -1) {};
\node (d1) at (1.5, -2) {};

\draw[fill= black] (a1) circle(0.6mm);
\draw[fill= black] (a2) circle(0.6mm);
\draw[fill= black] (a3) circle(0.6mm);
\draw[fill= black] (a4) circle(0.6mm);
\draw[fill= black] (b1) circle(0.6mm);
\draw[fill= black] (b2) circle(0.6mm);
\draw[fill= black] (b3) circle(0.6mm);
\draw[fill= black] (c1) circle(0.6mm);
\draw[fill= black] (c2) circle(0.6mm);
\draw[fill= black] (d1) circle(0.6mm);

\draw (a1.center) -- (b1.center);
\draw (a2.center) -- (b1.center);
\draw (a3.center) -- (b2.center);
\draw (a4.center) -- (b3.center);

\draw (b1.center) -- (c1.center);
\draw (b2.center) -- (c2.center);
\draw (b3.center) -- (c1.center);

\draw (c1.center) -- (d1.center);
\draw (c2.center) -- (d1.center);
\end{scope}

\begin{scope}[shift={({7*\Horspacing}, 0)}, scale=\Graphscaling]

\node (a1) at (0,1) {};
\node (a2) at (1,1) {};
\node (a3) at (2,1) {};
\node (a4) at (3,1) {};
\node (b1) at (0.5,0) {};
\node (b2) at (1.5,0) {};
\node (b3) at (2.5,0) {};
\node (c1) at (1, -1) {};
\node (c2) at (2, -1) {};
\node (d1) at (1.5, -2) {};

\draw[fill= black] (a1) circle(0.6mm);
\draw[fill= black] (a2) circle(0.6mm);
\draw[fill= black] (a3) circle(0.6mm);
\draw[fill= black] (a4) circle(0.6mm);
\draw[fill= black] (b1) circle(0.6mm);
\draw[fill= black] (b2) circle(0.6mm);
\draw[fill= black] (b3) circle(0.6mm);
\draw[fill= black] (c1) circle(0.6mm);
\draw[fill= black] (c2) circle(0.6mm);
\draw[fill= black] (d1) circle(0.6mm);

\draw (a1.center) -- (b1.center);
\draw (a2.center) -- (b2.center);
\draw (a3.center) -- (b1.center);
\draw (a4.center) -- (b3.center);

\draw (b1.center) -- (c1.center);
\draw (b2.center) -- (c2.center);
\draw (b3.center) -- (c1.center);

\draw (c1.center) -- (d1.center);
\draw (c2.center) -- (d1.center);
\end{scope}

\begin{scope}[shift={({8*\Horspacing}, 0)}, scale=\Graphscaling]

\node (a1) at (0,1) {};
\node (a2) at (1,1) {};
\node (a3) at (2,1) {};
\node (a4) at (3,1) {};
\node (b1) at (0.5,0) {};
\node (b2) at (1.5,0) {};
\node (b3) at (2.5,0) {};
\node (c1) at (1, -1) {};
\node (c2) at (2, -1) {};
\node (d1) at (1.5, -2) {};

\draw[fill= black] (a1) circle(0.6mm);
\draw[fill= black] (a2) circle(0.6mm);
\draw[fill= black] (a3) circle(0.6mm);
\draw[fill= black] (a4) circle(0.6mm);
\draw[fill= black] (b1) circle(0.6mm);
\draw[fill= black] (b2) circle(0.6mm);
\draw[fill= black] (b3) circle(0.6mm);
\draw[fill= black] (c1) circle(0.6mm);
\draw[fill= black] (c2) circle(0.6mm);
\draw[fill= black] (d1) circle(0.6mm);

\draw (a1.center) -- (b1.center);
\draw (a2.center) -- (b2.center);
\draw (a3.center) -- (b3.center);
\draw (a4.center) -- (b1.center);

\draw (b1.center) -- (c1.center);
\draw (b2.center) -- (c2.center);
\draw (b3.center) -- (c1.center);

\draw (c1.center) -- (d1.center);
\draw (c2.center) -- (d1.center);
\end{scope}

\begin{scope}[shift={({0*\Horspacing}, \Verspacing)}, scale=\Graphscaling]

\node (a1) at (0,1) {};
\node (a2) at (1,1) {};
\node (a3) at (2,1) {};
\node (a4) at (3,1) {};
\node (b1) at (0.5,0) {};
\node (b2) at (1.5,0) {};
\node (b3) at (2.5,0) {};
\node (c1) at (1, -1) {};
\node (c2) at (2, -1) {};
\node (d1) at (1.5, -2) {};

\draw[fill= black] (a1) circle(0.6mm);
\draw[fill= black] (a2) circle(0.6mm);
\draw[fill= black] (a3) circle(0.6mm);
\draw[fill= black] (a4) circle(0.6mm);
\draw[fill= black] (b1) circle(0.6mm);
\draw[fill= black] (b2) circle(0.6mm);
\draw[fill= black] (b3) circle(0.6mm);
\draw[fill= black] (c1) circle(0.6mm);
\draw[fill= black] (c2) circle(0.6mm);
\draw[fill= black] (d1) circle(0.6mm);

\draw (a1.center) -- (b2.center);
\draw (a2.center) -- (b1.center);
\draw (a3.center) -- (b1.center);
\draw (a4.center) -- (b3.center);

\draw (b1.center) -- (c1.center);
\draw (b2.center) -- (c2.center);
\draw (b3.center) -- (c1.center);

\draw (c1.center) -- (d1.center);
\draw (c2.center) -- (d1.center);
\end{scope}

\begin{scope}[shift={({1*\Horspacing}, \Verspacing)}, scale=\Graphscaling]

\node (a1) at (0,1) {};
\node (a2) at (1,1) {};
\node (a3) at (2,1) {};
\node (a4) at (3,1) {};
\node (b1) at (0.5,0) {};
\node (b2) at (1.5,0) {};
\node (b3) at (2.5,0) {};
\node (c1) at (1, -1) {};
\node (c2) at (2, -1) {};
\node (d1) at (1.5, -2) {};

\draw[fill= black] (a1) circle(0.6mm);
\draw[fill= black] (a2) circle(0.6mm);
\draw[fill= black] (a3) circle(0.6mm);
\draw[fill= black] (a4) circle(0.6mm);
\draw[fill= black] (b1) circle(0.6mm);
\draw[fill= black] (b2) circle(0.6mm);
\draw[fill= black] (b3) circle(0.6mm);
\draw[fill= black] (c1) circle(0.6mm);
\draw[fill= black] (c2) circle(0.6mm);
\draw[fill= black] (d1) circle(0.6mm);

\draw (a1.center) -- (b2.center);
\draw (a2.center) -- (b1.center);
\draw (a3.center) -- (b3.center);
\draw (a4.center) -- (b1.center);

\draw (b1.center) -- (c1.center);
\draw (b2.center) -- (c2.center);
\draw (b3.center) -- (c1.center);

\draw (c1.center) -- (d1.center);
\draw (c2.center) -- (d1.center);
\end{scope}

\begin{scope}[shift={({2*\Horspacing}, \Verspacing)}, scale=\Graphscaling]

\node (a1) at (0,1) {};
\node (a2) at (1,1) {};
\node (a3) at (2,1) {};
\node (a4) at (3,1) {};
\node (b1) at (0.5,0) {};
\node (b2) at (1.5,0) {};
\node (b3) at (2.5,0) {};
\node (c1) at (1, -1) {};
\node (c2) at (2, -1) {};
\node (d1) at (1.5, -2) {};

\draw[fill= black] (a1) circle(0.6mm);
\draw[fill= black] (a2) circle(0.6mm);
\draw[fill= black] (a3) circle(0.6mm);
\draw[fill= black] (a4) circle(0.6mm);
\draw[fill= black] (b1) circle(0.6mm);
\draw[fill= black] (b2) circle(0.6mm);
\draw[fill= black] (b3) circle(0.6mm);
\draw[fill= black] (c1) circle(0.6mm);
\draw[fill= black] (c2) circle(0.6mm);
\draw[fill= black] (d1) circle(0.6mm);

\draw (a1.center) -- (b2.center);
\draw (a2.center) -- (b3.center);
\draw (a3.center) -- (b1.center);
\draw (a4.center) -- (b1.center);

\draw (b1.center) -- (c1.center);
\draw (b2.center) -- (c2.center);
\draw (b3.center) -- (c1.center);

\draw (c1.center) -- (d1.center);
\draw (c2.center) -- (d1.center);
\end{scope}


\begin{scope}[shift={({3*\Horspacing}, \Verspacing)}, scale=\Graphscaling]

\node (a1) at (0,1) {};
\node (a2) at (1,1) {};
\node (a3) at (2,1) {};
\node (a4) at (3,1) {};
\node (b1) at (0.5,0) {};
\node (b2) at (1.5,0) {};
\node (b3) at (2.5,0) {};
\node (c1) at (1, -1) {};
\node (c2) at (2, -1) {};
\node (d1) at (1.5, -2) {};

\draw[fill= black] (a1) circle(0.6mm);
\draw[fill= black] (a2) circle(0.6mm);
\draw[fill= black] (a3) circle(0.6mm);
\draw[fill= black] (a4) circle(0.6mm);
\draw[fill= black] (b1) circle(0.6mm);
\draw[fill= black] (b2) circle(0.6mm);
\draw[fill= black] (b3) circle(0.6mm);
\draw[fill= black] (c1) circle(0.6mm);
\draw[fill= black] (c2) circle(0.6mm);
\draw[fill= black] (d1) circle(0.6mm);

\draw (a1.center) -- (b1.center);
\draw (a2.center) -- (b1.center);
\draw (a3.center) -- (b2.center);
\draw (a4.center) -- (b3.center);

\draw (b1.center) -- (c2.center);
\draw (b2.center) -- (c1.center);
\draw (b3.center) -- (c1.center);

\draw (c1.center) -- (d1.center);
\draw (c2.center) -- (d1.center);
\end{scope}

\begin{scope}[shift={({4*\Horspacing}, \Verspacing)}, scale=\Graphscaling]

\node (a1) at (0,1) {};
\node (a2) at (1,1) {};
\node (a3) at (2,1) {};
\node (a4) at (3,1) {};
\node (b1) at (0.5,0) {};
\node (b2) at (1.5,0) {};
\node (b3) at (2.5,0) {};
\node (c1) at (1, -1) {};
\node (c2) at (2, -1) {};
\node (d1) at (1.5, -2) {};

\draw[fill= black] (a1) circle(0.6mm);
\draw[fill= black] (a2) circle(0.6mm);
\draw[fill= black] (a3) circle(0.6mm);
\draw[fill= black] (a4) circle(0.6mm);
\draw[fill= black] (b1) circle(0.6mm);
\draw[fill= black] (b2) circle(0.6mm);
\draw[fill= black] (b3) circle(0.6mm);
\draw[fill= black] (c1) circle(0.6mm);
\draw[fill= black] (c2) circle(0.6mm);
\draw[fill= black] (d1) circle(0.6mm);

\draw (a1.center) -- (b1.center);
\draw (a2.center) -- (b2.center);
\draw (a3.center) -- (b1.center);
\draw (a4.center) -- (b3.center);

\draw (b1.center) -- (c2.center);
\draw (b2.center) -- (c1.center);
\draw (b3.center) -- (c1.center);

\draw (c1.center) -- (d1.center);
\draw (c2.center) -- (d1.center);
\end{scope}

\begin{scope}[shift={({5*\Horspacing}, \Verspacing)}, scale=\Graphscaling]

\node (a1) at (0,1) {};
\node (a2) at (1,1) {};
\node (a3) at (2,1) {};
\node (a4) at (3,1) {};
\node (b1) at (0.5,0) {};
\node (b2) at (1.5,0) {};
\node (b3) at (2.5,0) {};
\node (c1) at (1, -1) {};
\node (c2) at (2, -1) {};
\node (d1) at (1.5, -2) {};

\draw[fill= black] (a1) circle(0.6mm);
\draw[fill= black] (a2) circle(0.6mm);
\draw[fill= black] (a3) circle(0.6mm);
\draw[fill= black] (a4) circle(0.6mm);
\draw[fill= black] (b1) circle(0.6mm);
\draw[fill= black] (b2) circle(0.6mm);
\draw[fill= black] (b3) circle(0.6mm);
\draw[fill= black] (c1) circle(0.6mm);
\draw[fill= black] (c2) circle(0.6mm);
\draw[fill= black] (d1) circle(0.6mm);

\draw (a1.center) -- (b1.center);
\draw (a2.center) -- (b2.center);
\draw (a3.center) -- (b3.center);
\draw (a4.center) -- (b1.center);

\draw (b1.center) -- (c2.center);
\draw (b2.center) -- (c1.center);
\draw (b3.center) -- (c1.center);

\draw (c1.center) -- (d1.center);
\draw (c2.center) -- (d1.center);
\end{scope}

\begin{scope}[shift={({6*\Horspacing}, \Verspacing)}, scale=\Graphscaling]

\node (a1) at (0,1) {};
\node (a2) at (1,1) {};
\node (a3) at (2,1) {};
\node (a4) at (3,1) {};
\node (b1) at (0.5,0) {};
\node (b2) at (1.5,0) {};
\node (b3) at (2.5,0) {};
\node (c1) at (1, -1) {};
\node (c2) at (2, -1) {};
\node (d1) at (1.5, -2) {};

\draw[fill= black] (a1) circle(0.6mm);
\draw[fill= black] (a2) circle(0.6mm);
\draw[fill= black] (a3) circle(0.6mm);
\draw[fill= black] (a4) circle(0.6mm);
\draw[fill= black] (b1) circle(0.6mm);
\draw[fill= black] (b2) circle(0.6mm);
\draw[fill= black] (b3) circle(0.6mm);
\draw[fill= black] (c1) circle(0.6mm);
\draw[fill= black] (c2) circle(0.6mm);
\draw[fill= black] (d1) circle(0.6mm);

\draw (a1.center) -- (b2.center);
\draw (a2.center) -- (b1.center);
\draw (a3.center) -- (b1.center);
\draw (a4.center) -- (b3.center);

\draw (b1.center) -- (c2.center);
\draw (b2.center) -- (c1.center);
\draw (b3.center) -- (c1.center);

\draw (c1.center) -- (d1.center);
\draw (c2.center) -- (d1.center);
\end{scope}

\begin{scope}[shift={({7*\Horspacing}, \Verspacing)}, scale=\Graphscaling]

\node (a1) at (0,1) {};
\node (a2) at (1,1) {};
\node (a3) at (2,1) {};
\node (a4) at (3,1) {};
\node (b1) at (0.5,0) {};
\node (b2) at (1.5,0) {};
\node (b3) at (2.5,0) {};
\node (c1) at (1, -1) {};
\node (c2) at (2, -1) {};
\node (d1) at (1.5, -2) {};

\draw[fill= black] (a1) circle(0.6mm);
\draw[fill= black] (a2) circle(0.6mm);
\draw[fill= black] (a3) circle(0.6mm);
\draw[fill= black] (a4) circle(0.6mm);
\draw[fill= black] (b1) circle(0.6mm);
\draw[fill= black] (b2) circle(0.6mm);
\draw[fill= black] (b3) circle(0.6mm);
\draw[fill= black] (c1) circle(0.6mm);
\draw[fill= black] (c2) circle(0.6mm);
\draw[fill= black] (d1) circle(0.6mm);

\draw (a1.center) -- (b2.center);
\draw (a2.center) -- (b1.center);
\draw (a3.center) -- (b3.center);
\draw (a4.center) -- (b1.center);

\draw (b1.center) -- (c2.center);
\draw (b2.center) -- (c1.center);
\draw (b3.center) -- (c1.center);

\draw (c1.center) -- (d1.center);
\draw (c2.center) -- (d1.center);
\end{scope}

\begin{scope}[shift={({8*\Horspacing}, \Verspacing)}, scale=\Graphscaling]

\node (a1) at (0,1) {};
\node (a2) at (1,1) {};
\node (a3) at (2,1) {};
\node (a4) at (3,1) {};
\node (b1) at (0.5,0) {};
\node (b2) at (1.5,0) {};
\node (b3) at (2.5,0) {};
\node (c1) at (1, -1) {};
\node (c2) at (2, -1) {};
\node (d1) at (1.5, -2) {};

\draw[fill= black] (a1) circle(0.6mm);
\draw[fill= black] (a2) circle(0.6mm);
\draw[fill= black] (a3) circle(0.6mm);
\draw[fill= black] (a4) circle(0.6mm);
\draw[fill= black] (b1) circle(0.6mm);
\draw[fill= black] (b2) circle(0.6mm);
\draw[fill= black] (b3) circle(0.6mm);
\draw[fill= black] (c1) circle(0.6mm);
\draw[fill= black] (c2) circle(0.6mm);
\draw[fill= black] (d1) circle(0.6mm);

\draw (a1.center) -- (b2.center);
\draw (a2.center) -- (b3.center);
\draw (a3.center) -- (b1.center);
\draw (a4.center) -- (b1.center);

\draw (b1.center) -- (c2.center);
\draw (b2.center) -- (c1.center);
\draw (b3.center) -- (c1.center);

\draw (c1.center) -- (d1.center);
\draw (c2.center) -- (d1.center);
\end{scope}
\end{tikzpicture}
\end{figure}

Finally, $e^C: S(V) \rightarrow S(V)$ is defined as
\begin{equation}\label{eqn:exp_coder}
e^C = \textrm{Id} + C + \frac{C^2}{2!} + \frac{C^3}{3!} + \cdots
\end{equation}
Notice that $C^n$ vanishes for all elements of $S^{\leq n}(V)$. As such, the convergence of $e^C$ is guaranteed by the fact that for any fixed element of $S(V)$, the right side of Equation \ref{eqn:exp_coder} has finitely many nonzero terms.

We now apply this $L_\infty$ formalism to analyze the DGLA structures from Section~\ref{dgladir}. To do so, we take $V=\Omega_M[1]$, the de Rham complex of the Dirac structure $M$ with grading shifted so that $\Omega^1_M$ occurs in degree zero. Given two Dirac structures $L$ and $L_{\eps}$ transverse to $M$, let $\iota_\varepsilon$ be the coderivation of $S(V[1]) = S(\Omega_M[2])$ induced by the contraction map 
$\Omega_M \to \Omega_M $, $\alpha \mapsto \iota_\varepsilon\alpha$, and let $\mu$ be the coderivation of $S(\Omega_M[2])$ induced by the wedge product, $\alpha \odot \beta \mapsto  \alpha \wedge \beta$. Define 
\begin{equation}\label{eq:R_eps_definition}
 R_\varepsilon=[\iota_\varepsilon,\mu],
\end{equation}
where the bracket stands for the graded commutator of coderivations.
We will show in Section \ref{section:proof} that the map $e^{R_\varepsilon}$ provides the required $L_{\infty}$-isomorphism between the DGLAs $(\Omega_M[1], d_M, [\cdot , \cdot]_L)$ and $(\Omega_M[1], d_M, [\cdot ,\cdot ]_{L_{\varepsilon}})$. In Section \ref{section:action} we verify that this $L_{\infty}$-isomorphism acts on Maurer-Cartan elements precisely as given in Equation~\ref{keymcmap}.

\section{Construction of $L_\infty$ structures and morphisms}\label{section:proof}

\subsection{Alternative grading on differential forms}

Sections of $\TT X$ act on differential forms $\Omega$ by the formula $(X+\xi)\cdot \rho = \iota_X\rho + \xi\wedge\rho$, which extends to an action of the Clifford algebra generated by $(\mathbb{T}X, \langle \cdot, \cdot \rangle)$ on $\Omega$. Given an almost Dirac structure $M$, the differential forms annihilated by the Clifford action of $M$ form a line subbundle $K_M$ of $\Omega$ called the \emph{pure spinor line} of $M$. This pure spinor line completely encodes $M$; see~\cite{Gua11} for a detailed description of the equivalence between Dirac structures and pure spinors.  
Also detailed there is the fact that a pair $M, L$ of transverse almost Dirac structures induces an alternative $\mathbb{Z}$-grading on the module 
$\Omega$; that is, on an $n$-manifold, $\Omega = \oplus_{k=0}^n U_k$, where $U_k = (\wedge^kL) \cdot K_M$.  
This yields the familiar Fock space description of the irreducible spinor module, with $L$ and $M$ acting via raising and lowering operators, respectively:
\begin{equation}\label{eq:spinor_grading}
\xymatrix{
K_M=U_0~~~~ \ar@/^1.0pc/@[black][r]^{L}
  & U_1   \ar@/^1.0pc/@[black][l]^{M} \ar@/^1.0pc/@[black][r]^{L}
  & U_2  \ar@/^1.0pc/@[black][l]^{M} \ar@/^1.0pc/@[black][r]^{L}
  & ~~~\dots \ar@/^1.0pc/@[black][l]^{M}
  & \dots \ar@/^1.0pc/@[black][r]^{L} ~~~
  & ~~~~U_n = K_L \ar@/^1.0pc/@[black][l]^{M}.
  }
\end{equation}
Note that the duality pairing between $L$ and $M$ gives canonical identifications $U_k = \wedge^kM^*\otimes K_M = \wedge^{n-k}L^* \otimes K_L$. 

The differential forms $\Omega$ are equipped with a canonical operator, the twisted exterior derivative $d_H = d + H\wedge\cdot$, which generates the Courant bracket 
according to the derived bracket formula~(see \cite{KS04} for details)
\begin{equation}\label{derbr}
[\![x,y]\!]_H \cdot \rho = [[d_H,x\cdot ],y\cdot ] \rho,\qquad \rho\in\Omega.
\end{equation}
Using this $\mathbb{Z}$-grading, we may decompose the operator $d_H$ into the sum of its graded pieces. Without assuming involutivity for $L$ or $M$, the twisted de Rham differential can only have components of degree $-3$, $-1$, $1$, and $3$,  
which we denote by $N_L$, $\del$, $\delbar$, and $N_M$, respectively:
\begin{equation}\label{dcomps}
\begin{array}{lllllllllllll}
d_H &=& d_{-3}&+&d_{-1}&+&d_1&+&d_3 \\
&=&N_L &+& \del &+& \delbar &+& N_M.
\end{array}
\end{equation}
Here $N_M$ (resp. $N_L$) is the Nijenhuis tensor of $M$ (resp. $L$), i.e. $N_M \in \Omega^3_M$ is defined by 
$N_M(x,y,z) = \langle [\![x,y ]\!]_H, z \rangle$, $x,y,z \in \Gamma(M)$. 
This tensor measures the failure of the almost Dirac structure to be involutive: $d_3=0$ (resp. $d_{-3}=0$) if and only if $M$ (resp. $L$) 
is a Dirac structure (see \cite[Theorem 2.9]{Gua11}).

\subsection{Construction of $L_\infty$ structure using $BV_\infty$ torsor}

We now use the alternative grading on differential forms described above to construct the symmetric brackets on $\Omega_M[2]$ required for defining an $L_\infty$ structure on $\Omega_M[1]$. Rather than viewing the space of differential forms as a graded algebra as usual, we view it as a graded module for the action of the graded algebra $\Omega_M$, by identifying $\Omega_M$ with the $L$-multivectors $\mathcal{X}_L$ and using the Clifford action described above.  

Recall that if $P$ is a graded module over a graded algebra $A$, a linear map $D:P \to P$ is called a $k^{th}$ order differential operator if $ad_{x_k}\dots ad_{x_0}(D)=0$ for any $x_0, \dots, x_k \in A$, where $ad_x(-)$ is the graded commutator $[-,x]$. Now consider the twisted exterior derivative $d_H$ acting on the $\Omega_M$--module $\Omega$ of differential forms: with respect to this algebra action,  $d_H$ is not guaranteed to be a first order differential operator. However, the identity~\eqref{derbr} implies that  $ad_{x_3}\dots ad_{x_0}(d_H)=0$ for any quadruple of sections $x_0, \dots, x_3$ of $\TT X$; that is, $d_H$ is a $3^{rd}$ order differential operator with respect to the module structure over $\Omega_M$. More precisely, we can say the following.

\begin{lemma} \label{trading_degree_for_order}
The $(3-2k)^{th}$ graded piece of $d_H$ is a $k$-th order differential operator on $\Omega$, viewed as a module over the algebra $\Omega_M$.
\end{lemma}

\begin{proof}
The key point is the following claim: if a linear operator $D$ on $\Omega$ is a $k^{th}$ order operator with respect to the Clifford action of the algebra of $M$-multivectors 
$\mathcal{X}_M$ on $\Omega$, then the graded pieces $D_i$, $i>k$ 
all vanish. Assuming this claim is true, we finish the proof by applying the claim successively to $D=[d_H,x_0]$ for $k=2$, $D=[[d_H,x_0],x_1]$ for $k=1$, and $D=[[[d_H,x_0],x_1],x_2]$ for $k=0$, where the $x_i$ are elements of $\Omega^1_M$.

To prove the claim, we fix $i>k$ and prove that $D_i$ vanishes on each $U_j$ by induction on $j$. On $U_0$, by contradiction suppose $D_i(\rho)\not= 0$ for some $\rho \in U_0$. Then one can find $x_1,\dots, x_i\in \mathcal{X}^1_M$ such that $x_i\dots x_1 (D\rho) \not=0$. 
Since every $x\in \mathcal{X}^1_M$ annihilates $\rho$, we get $(ad_{x_i}\dots ad_{x_1} D)\rho =x_i\dots x_1 (D\rho)  \not=0$, $i>k$, which contradicts the assumption that $D$ is a $k^{th}$ order differential operator on $\Omega$ viewed as $\mathcal{X}_M$-module (where $\mathcal{X}_M$ is equipped with the opposite grading). 
Proving this step of the induction is done using the same argument as the base case.
\end{proof}

We find it convenient to axiomatize the structure of the $\Omega_M$--module $\Omega$ described above in the following way:

\begin{definition}
An \textbf{almost $BV_\infty$ torsor} $(A, P, \Delta)$ consists of:
\begin{enumerate}
\item $\mathbb{Z}$-graded commutative algebra $A$,
\item a free, rank $1$ graded $A$-module $P$, 
\item an odd map $\Delta:P\to P$ such that for each $k$ the graded piece $\Delta_{3-2k}$ is a $k^{th}$ order differential operator on the $A$-module $P$.
\end{enumerate}
We call $(A, P, \Delta)$ a \textbf{$BV_\infty$ torsor} if $\Delta^2=0$.
\end{definition} 

Given an almost $BV_\infty$ torsor $(A, P, \Delta)$, one can construct a family of symmetric brackets, for $k\ge0$, on $A[2]$ given by the formulas
\begin{equation}\label{eq:quantized_derived_brackets}
m_k^\Delta(x_1, x_2, \dots, x_k) \cdot \rho = [\dots [[\Delta_{3-2k}, x_1], x_2]\dots x_k]  \rho,
\end{equation}
for any $x_1, \dots, x_k \in A$, $\rho \in P$. 
Note that by Lemma~\ref{trading_degree_for_order}, the right hand side of~\eqref{eq:quantized_derived_brackets} is necessarily an order $0$ operator on $P$, and thus is just a multiplication by an element 
in $A$. 
So, the brackets do not depend on $\rho$. Let us denote the corresponding coderivation on $S(A[2])$ by $L_\infty(A , P, \Delta)$, or by $L_\infty(\Delta)$ for short.  Finally, the condition $\Delta^2=0$ implies the required condition $L_\infty(\Delta)^2=0$, as we now show.

\begin{theorem} \label{quantized_derived_brackets_integrability}
If $(A,P,\Delta)$ is a $BV_\infty$ torsor, i.e. if $\Delta^2 =0$, then the brackets $m_k^\Delta$ given above define an $L_\infty$ algebra structure on $A[1]$.
\end{theorem}
\begin{proof}
For $k\ge1$, let $J_k^\Delta(x_1,\dots,x_k)$ denote the $k^{th}$ Jacobiator, i.e. the left hand side of the generalized Jacobi equations \eqref{eq:generalized_Jacobi}, for the $L_\infty$ structure given by the brackets $m_k^\Delta$.
It is a standard computation (see e.g. \cite{BDA,Vor}) 
to check that 
\begin{equation}\label{jacobi}
J_k^\Delta(x_1, \dots, x_k) = ad_{x_k} \dots ad_{x_1} \left((\Delta^2)_{4-2k}\right).
\end{equation}
Note that the $(4-2k)$-th graded piece of $\Delta^2=\frac{1}{2}[\Delta,\Delta]$ is a $k^{th}$ order differential operator, so the quantity on the right hand side of~\eqref{jacobi} defines an element in $A$.
Therefore, the identity $\Delta^2 =0$ guarantees vanishing of all the Jacobiators $J_k^\Delta$, $k\ge1$.
\end{proof}

\begin{remark}
The construction procedure above is a version of Voronov's derived bracket construction~\cite{Vor}, described also in the papers 
\cite{BDA,Kra00,KS04}. The particular formalism we are using is quite close to the notion of $BV_\infty$-algebra (see \cite[Definition 1.2]{BaVo}), with the following slight modifications: 
\begin{itemize}
 \item[--] We allow the $L_\infty$ algebra to have curvature, i.e. a zero-arity bracket, so the operator $\Delta$ is allowed to have a component of degree $+3$.
 \item[--] More importantly, the operator $\Delta$ does not act on the algebra $A$ itself, but rather on a free rank $1$ $A$-module $P$. 
\end{itemize}
\end{remark} 

\begin{remark}
The $L_\infty$ structure obtained by this procedure is compatible with the wedge product on $\Omega_M$, giving rise to a
$G_\infty$-algebra (see \cite{GV95,Tam98}). 
\end{remark} 

\begin{example}[$L_\infty$ algebra structure on de Rham complex of a Dirac structure]
Our main application of the derived bracket construction is to the $BV_\infty$ torsor $(\Omega_M\cong \mathcal{X}_L , \Omega, d_H)$ associated to the almost Dirac structure $M$ and the complementary almost Dirac structure $L$. The resulting $L_\infty$ algebra structure on $\Omega_M[1]$ may be written in terms of a sequence of symmetric brackets on $\Omega_M[2]$ as described in Section~\ref{linfsym}, and these may be written in terms of the graded components of the twisted de Rham differential $d_H$ as follows:
\begin{equation}\label{eq:derived_brackets}
\begin{array}{lll}
m_0 \cdot \rho &= N_M \cdot \rho, \\
m_1(\alpha) \cdot \rho &= [ \delbar, \alpha ] \cdot \rho, \\
m_2(\alpha, \beta) \cdot \rho &= [[ \del, \alpha ], \beta] \cdot \rho, \\
m_3(\alpha, \beta, \gamma) \cdot \rho &= [[[ N_L, \alpha ], \beta], \gamma ]\cdot \rho. 
\end{array}
\end{equation}
Note that we get a flat $L_\infty$ algebra on $\Omega_M[1]$ precisely when $N_M$ vanishes, i.e. when  $M$ is involutive.  Also, when both $L$ and $M$ are integrable, $m_3$ vanishes and we obtain the DGLA structure 
$(\Omega_M[1], d_M, [\cdot,\cdot]_L)$ described in Section~\ref{dgladir}.
\end{example}

\subsection{Construction of $L_\infty$ morphism using $BV_\infty$ torsor}

Finally, consider a pair $L, L'$ of almost Dirac structures transverse to $M$. We wish 
to compare the corresponding pair of $L_\infty$ structures obtained on $\Omega_M[1]$ from the construction in the previous section.
We continue to use the notation $\Omega$ for the differential forms on $X$ equipped with the alternative grading coming from the splitting $\TT X = M\oplus L$, and we use  $\Omega'$ to denote the grading derived from the splitting $\TT X = M\oplus L'$. The two $L_\infty$ algebras in question arise from the $BV_\infty$ torsors $(\Omega_M\cong \mathcal{X}_L , \Omega, d_H)$ and $(\Omega_M\cong \mathcal{X}_{L'} , \Omega', d_H)$. 

Fix $\varepsilon\in\mathcal{X}^2_M$ such that $L'$ is the graph of $\varepsilon$ viewed as a skew map $L\to M$. The tensor $\eps$ acts on $\Omega$ via the Clifford action; exponentiating this, we obtain an isomorphism  $e^{-\varepsilon }:\Omega' \to \Omega$ between the $BV_\infty$ torsors $(\Omega_M, \Omega', d_H)$ and $(\Omega_M, \Omega, e^{-\varepsilon}d_He^{\varepsilon })$, in the sense that $e^{-\varepsilon}$ is an equivariant map of $\Omega_M$-modules and the following diagram commutes:
$$
\xymatrix{
\Omega' \ar[r]^{ e^{-\varepsilon} }\ar[d]_-{d_H}& \Omega\ar[d]^-{ e^{-\varepsilon}d_{H}e^{\varepsilon }}\\
\Omega' \ar[r]^{e^{-\varepsilon}} & \Omega
}
$$
It it transparent from the definitions that isomorphic $BV_\infty$ torsors result in identical $L_\infty$ structures via Theorem~\ref{quantized_derived_brackets_integrability}. Therefore, we may trade the $BV_\infty$ torsor $(\Omega_M, \Omega', d_H)$ for the isomorphic $BV_\infty$ torsor $(\Omega_M, \Omega, e^{-\varepsilon}d_He^{\varepsilon })$, and now our question is reduced to the comparison of $L_\infty$ structures $L_\infty(\Omega_M, \Omega, e^{-\varepsilon}d_{H}e^{\varepsilon })$ and $L_\infty(\Omega_M, \Omega, d_H)$.

\begin{theorem}\label{Thm:derived_bracket_construction_preserves_gauge_action}
Let $\varepsilon \in \mathcal{X}^2_M$.
There is a canonical $L_\infty$-isomorphism between the $L_\infty$ algebras $L_\infty(\Omega_M, \Omega, e^{-\varepsilon}d_He^{\varepsilon })$ and $L_\infty(\Omega_M, \Omega, d_H)$, given by $e^{R_{\varepsilon}}$, the exponential of the coderivation $R_\varepsilon$ given by the commutator $[\iota_\varepsilon,\mu]$ (see Equation~\ref{eqn:exp_coder}). In other words, we have 
\begin{equation}\label{epsilon_vs_R}
L_\infty(e^{-\varepsilon} d_H e^{\varepsilon}) = e^{-R_\varepsilon} L_\infty(d_H) e^{R_{\varepsilon}}.
\end{equation}
\end{theorem}

\begin{proof}
We can expand the both sides of \eqref{epsilon_vs_R} as a series of iterated adjoint actions. So, it is enough to show, that $L_\infty([d_H,\varepsilon])=[L_\infty(d_H),R_\varepsilon]$ (note that since $\varepsilon$ is a $2^{nd}$ order degree $-2$ operator on $\Omega$, the tuple $(\Omega_M, \Omega, [d_H,\varepsilon])$ forms an almost $BV_\infty$ torsor). Checking the latter identity is a straightforward computation and is left to the reader.\end{proof}


\begin{corollary}\label{cor:independence_of_L_infty_from_the_choice_of_complement}
 Let $M$ and $L$ be a pair of transverse maximal isotropic subbundles of an exact Courant algebroid $E$, 
and let $L'=e^\varepsilon L$ be another maximal isotropic subbundle transverse to $M$. 
Then the $L_\infty$ structures on $\Omega_M[1]$ given by the pairs $(M,L')$ and $(M,L)$ are $L_\infty$-isomorphic via the $L_\infty$ map $e^{R_\varepsilon}$.
\end{corollary}

\begin{remark}
The theorem above is closely related to the results of \cite{CaSch}, where the authors establish a sufficient condition for
equivalence of $L_\infty$ structures produced by Voronov's method. They are able to apply their general condition to obtain a version of the above theorem for regular Dirac structures, i.e. Dirac structures $M$  for which $M\cap T^*X$ has constant rank. 
Our method is different, using the notion of $BV_\infty$ torsor, and we conclude the more general result for arbitrary (and possibly non-involutive) Dirac structures.  
\end{remark}

\section{Action of the {$L_\infty$} map on Maurer-Cartan elements}\label{section:action}
Suppose $M$ is a Dirac structure and $L, L'$ are Dirac structures transverse to $M$, giving rise to two DGLA structures $(\Omega_M[1], d_M, [\cdot ,\cdot ]_{L})$ and 
$(\Omega_M[1], d_M, [\cdot , \cdot]_{L'})$.  As before, let $\varepsilon\in\mathcal{X}^2_M$ such that $L'$ is the graph of $\varepsilon$ viewed as a skew map $L\to M$.
Theorem~\ref{Thm:derived_bracket_construction_preserves_gauge_action} and Corollary~\ref{cor:independence_of_L_infty_from_the_choice_of_complement} provide a $L_\infty$-isomorphism from the former to the latter DGLA, given by the exponential of the coderivation $R_\varepsilon=[\iota_\varepsilon,\mu]$, where $\iota_{\varepsilon}$ and $\mu$ are the coderivations of 
$S(\Omega_M[2])$ given by contraction with $\varepsilon$ and the wedge product on $\Omega_M$, respectively.  In this section we compute the action of this $L_\infty$--isomorphism on Maurer-Cartan elements.

\begin{lemma}\label{R_explicit}
For $2$-forms $\omega_1,\omega_2 \in \Omega_M^2$, we have
$$
R_\varepsilon (\omega_1 \odot \omega_2) = \omega_1 \varepsilon \omega_2 + \omega_2 \varepsilon \omega_1,
$$
where on the right hand side we view $\omega_1,\omega_2$ as maps $M\to L$, and $\varepsilon$ as a map $L\to M$.
\end{lemma}

\begin{proof} 
Computing the action of $R_\varepsilon$, we obtain
\[
R_\varepsilon (\omega_1 \odot \omega_2) = i_\varepsilon(\omega_1)\wedge \omega_1 + \omega_1 \wedge i_\varepsilon(\omega_2) 
- i_\varepsilon(\omega_1\wedge \omega_2).
\]
We then use the identity $i_\varepsilon(\omega_1\wedge \omega_2) = (i_\varepsilon \omega_1) \omega_2 + (i_\varepsilon \omega_2) \omega_1 - \omega_2\varepsilon\omega_1 - \omega_1\varepsilon\omega_2$, 
obtained by evaluating the left hand side on $X\in M$:
\begin{align*}
i_Xi_\varepsilon(\omega_1\wedge \omega_2) &= i_\varepsilon((i_X \omega_1) \wedge \omega_2 +  (i_X \omega_2)\wedge \omega_1)\\
&= (i_X \omega_1) (i_\varepsilon \omega_2) + (i_X \omega_2) (i_\varepsilon \omega_1) - i_{(i_{\omega_1(X)} \varepsilon)} \omega_2 - i_{(i_{\omega_2(X)} \varepsilon)} \omega_1 \\
&= i_X((i_\varepsilon \omega_2) \omega_1 + (i_\varepsilon \omega_1)\omega_2 - \omega_2\varepsilon\omega_1 - \omega_1\varepsilon\omega_2).
\end{align*}
\end{proof}

\begin{theorem}\label{mccomp}
Let $\omega$ be a Maurer-Cartan element in the DGLA $(\Omega_M[1], d_M, [~,~]_L)$. 
Let $\varepsilon=\varepsilon_1t+\varepsilon_2t^2+... \in t\Omega^2_L[[t]]$ define a formal
deformation of the Dirac structure $L$.

Then the $L_\infty$ map $f=e^{R_\varepsilon}$ sends $\omega$ to $B =  \sum_{n=0}^\infty \omega (\varepsilon \omega)^n$. 
Moreover, if the series $\sum_{n\ge1} \varepsilon_nt^n$ is convergent and $t$ is so small that $Id-\varepsilon \omega$ is invertible, 
then the resulting series for $B$ will be convergent and $B = \omega (Id-\varepsilon \omega)^{-1}$.
\end{theorem}

\begin{proof}

Let $V:=\Omega_M[2]$, and let us denote the component $S^n(V)\to V$ of the coalgebra automorhism $e^{R_\varepsilon}$ by $f_n$. By definition
$$
e^{R_\varepsilon}(\omega)=\sum_{n\ge 1} \frac{1}{n!} f_n(\omega^{\odot n}).
$$

Note that $R_\varepsilon:S^2V \to V$, so $f_1=Id$. 
Next, Lemma \ref{R_explicit} implies that for an $n$-tuple of $2$-forms $\omega_i \in \Omega_M$, $n\ge2$
\begin{equation}\label{R_explicit_upgrade}
f_n(\omega_1 \odot \dots \odot \omega_n) = \
\sum_{\sigma \in S_n} \omega_{\sigma(1)} \varepsilon \omega_{\sigma(2)} \varepsilon \dots \varepsilon 
\omega_{\sigma(n)}.
\end{equation}

So, we see that 
$$
e^{R_\varepsilon}(\omega)=\sum_{n\ge 1} \frac{1}{n!} f_n(\omega^{\odot n}) = \omega + \omega \varepsilon \omega +
\omega \varepsilon \omega \varepsilon \omega + \dots
$$
\end{proof}
%
\begin{remark}
We see from the above theorem that the natural correspondence between Maurer-Cartan elements~\eqref{keymcmap} suggested by Dirac geometry does indeed coincide with the application of an $L_\infty$ morphism. 
\end{remark}

\section{Examples}\label{section:examples}

\subsection{Poisson structures}\label{poisex}
To a Poisson manifold $(X,\pi)$ there corresponds a natural DGLA on the (shifted) de Rham complex $(\Omega[1], d, [~,~]_\pi)$, 
which is sometimes called the Koszul DGLA. To obtain the DGLA from our formalism, one needs to consider the pair of transverse Dirac structures
$M=T$ and $L_\pi=graph(\pi:T^*\to T)$. Then the Koszul DGLA coincides with the DGLA $(\Omega_M[1], d_M, [-,-]_{L_\pi})$ introduced in Section 
\ref{dgladir}. Another natural choice of complement to $M$ is $L=T^*$; the DGLA $(\Omega_M[1], d_M, [~,~]_L)$ is the de Rham complex $\Omega$ with zero bracket.
Applying Corollary \ref{cor:independence_of_L_infty_from_the_choice_of_complement}, we recover the corresponding results of \cite{ShTa08} and \cite{FM12}.

\begin{theorem}\label{Thm:Formality_of_Koszul_DGLA}
Let $(X,\pi)$ be a Poisson manifold.
Then the Koszul DGLA $(\Omega[1], d, [~,~]_\pi)$ is formal, that is, $L_\infty$-isomorphic to the abelian DGLA $(\Omega[1], d, [~,~]=0)$.
\end{theorem}

\subsection{Non-integrable bivectors. Quasi-Poisson manifolds}

We may extend the above result to the case of general bivectors which may not be Poisson. 
The proof of the following result is identical to that of Theorem \ref{Thm:Formality_of_Koszul_DGLA}, but due to the fact that the Lagrangian  
$graph(\pi:T^*\to T)$ has nonzero Nijenhuis tensor, we obtain a nontrivial ternary bracket.

\begin{theorem}\label{thm:non_integrable_bivectors}
 Let $X$ be a manifold and $\pi$ any smooth bivector field. Then we obtain a $L_\infty$ structure $(\Omega[1], \ell_1,\ell_2,\ell_3)$ on the shifted de Rham complex of $X$, where $\ell_1=d$ is the de Rham differential, $\ell_2$ is given by the Koszul bracket 
\begin{equation}\label{luzsok}
\ell_2(x,y) = (-1)^{|{x}|}(\mathcal{L}_\pi(x \wedge y) - \mathcal{L}_\pi(x) \wedge y ) -  x \wedge \mathcal{L}_\pi(y),
\end{equation}
and $\ell_3$ is given by 
\begin{equation}\label{yranret}
\ell_3(x,y,z)=i_{\frac{1}{2}[\pi,\pi]}(x \wedge y \wedge z)
\end{equation}
for $1$-forms $x,y,z$, and extended to all forms by requiring
the Leibniz rule in each entry.

Furthermore, 
this $L_\infty$ algebra is formal, i.e. isomorphic to the abelian $L_\infty$ algebra $(\Omega[1], d)$ via the $L_\infty$ map $e^{R_\pi}$.
\end{theorem}

An important source of non-integrable bivectors is quasi-Poisson geometry~\cite{AKSM}, prominent in the theory of group valued moment maps \cite{AMM}
and useful for studying the moduli space of flat connections on surfaces \cite{LiBS}.
Let $G$ be a compact Lie group acting on a manifold $X$. Fix an invariant inner product $\langle\cdot,\cdot\rangle$ on the Lie algebra $\mathfrak{g}$ and form the Cartan $3$-form
$\eta\in \wedge^3 \mathfrak{g}^*$ by letting $\eta(x,y,z)=\frac{1}{2}\langle[x,y],z\rangle$. 
A bivector $\pi$ is called \textit{quasi-Poisson} if $\pi$ is $G$-invariant and $[\pi,\pi]=(\eta^{-1})_X$, where the inverse is taken with respect to
the inner product, and the index $X$ indicates that the $3$-vector $\eta^{-1}$ is transported onto $X$ via the action map.

In the setting of a quasi-Poisson manifold $(G,X,\pi)$, therefore, it makes sense to restrict the cubic $L_\infty$ algebra obtained above to the subspace of $G$-invariant forms $(\Omega)^G$. 
All the brackets as well as the formality morphism $e^{R_\pi}$ respect the restriction, so we obtain the following result.

\begin{corollary}\label{cor:non_integrable_bivectors}
Any quasi-Poisson manifold $(G, X, \pi)$ gives rise to a cubic $L_\infty$ structure 
$((\Omega)^G[1], d, \ell_2, \ell_3)$ on the shifted $G$-invariant differential forms, where $\ell_2,\ell_3$ are given by \eqref{luzsok} and \eqref{yranret}. Furthermore, this $L_\infty$ algebra is formal.
\end{corollary}

\subsection{$L_\infty$-automorphisms of the Lie algebra of multivector fields}

On any manifold $X$, let $M$, $L$ be the Dirac structures $T^*X$, $TX$, respectively,  for the Courant bracket with $H=0$.  Then let $L'$ be the Dirac structure defined by the graph of a closed 2-form $\omega\in \Omega^2$. Since the ``B-field gauge transformation'' 
\begin{equation}
e^\omega: X+\xi \mapsto X + \xi + \iota_X \omega
\end{equation}
is an automorphism of the Courant algebroid $\TT X$ taking $L$ to $L'$, the transversals $L$ and $L'$ define identical DGLA structures on the de Rham complex of $T^*X$: in both cases we obtain the complex of multivector fields $\mathcal{X}$ with zero differential and bracket given by the Schouten bracket. 

Applying Corollary~\ref{cor:independence_of_L_infty_from_the_choice_of_complement} to this example, we see immediately that $\exp(R_\omega)$ defines an $L_\infty$-automorphism of the Lie algebra of multivector fields:

\begin{theorem}
 To every closed $2$-form $\omega$ there corresponds an $L_\infty$-automorhism $e^{R_\omega}$ of the Lie algebra
$(\mathcal{X}[1], [-,-]_S)$ of multivector fields equipped with Schouten bracket.
\end{theorem}
%
Let us compute the action of $e^{R_\omega}$ on Maurer-Cartan elements of the Lie algebra $(\mathcal{X}[1], [-,-]_S)$, i.e. on Poisson bivectors. 
Theorem \ref{mccomp} states that if $\pi$ is a Poisson bivector, 
then $(e^{R_\omega})_*\pi = \pi (1-\omega \pi)^{-1}= \pi +  \pi \omega \pi + \pi \omega \pi \omega \pi + \dots $ is again a Poisson bivector, 
provided the series converges. The latter series is well-known in the Poisson literature as a gauge, or $B$-field, transformation of $\pi$ (see~\cite{Sevwein}). Geometrically, the gauge transformation of a Poisson structure $\pi$ corresponds to subtracting the pullback of $\omega$ from the symplectic form on each symplectic leaf of $\pi$.

\subsection{Dirac structures on Lie groups}
\label{Subsec:lie_groups}

Let $\mathfrak{g}$ be a quadratic Lie algebra, that is a Lie algebra with invariant non-degenerate symmetric product $\langle\cdot,\cdot\rangle$. 
Let $\overline{\mathfrak{g}}$ be the same Lie algebra $\mathfrak{g}$ with the symmetric product $-\langle\cdot,\cdot\rangle$. 
Then we can form the ``double'' $\mathfrak{d}=\mathfrak{g}\oplus \overline{\mathfrak{g}}$, 
endowed with a unique Lie bracket extending the brackets on $\mathfrak{g}$ and $\overline{\mathfrak{g}}$ 
and compatible with the direct sum of inner products.
One can think of $\mathfrak{d}$ as a finite-dimensional version of the Courant algebroid $\TT G = T_G \oplus T^*_G $, 
where $G$ is a Lie group integrating $\mathfrak{g}$.
Specifically, there is a natural isomorphism (see~\cite{ABM})
$\TT G \cong G\times (\mathfrak{g}\oplus \overline{\mathfrak{g}})$, 
preserving the pairing and the bracket, where the generalized tangent bundle $\TT G$ is endowed with the Courant bracket
twisted by the Cartan $3$-form.

The diagonal $\Delta=\{(x,x):x\in \mathfrak{g}\}$ is a Lagrangian subalgebra in $\mathfrak{d}=\mathfrak{g}\oplus \overline{\mathfrak{g}}$. The antidiagonal
$\overline{\Delta}=\{(x,-x): x\in \mathfrak{g}\}$ is also Lagrangian, but is not a subalgebra, since $[\overline{\Delta},\overline{\Delta}]\subset \Delta$.
Applying Corollary \ref{cor:independence_of_L_infty_from_the_choice_of_complement} to the pair of Lagrangians $(\Delta,\overline{\Delta})$ inside $\mathfrak{d}$ we get the following
\begin{theorem}\label{Thm:cartan_vs_gauss_three_term_L_infty}
Let $\mathfrak{g}$ be a quadratic Lie algebra. Then $(\wedge \mathfrak{g}^*)[1]$ carries a natural cubic $L_\infty$ structure, where $\ell_1 = d_{CE}$ is the Chevalley-Eilenberg differential, $\ell_2$ is the zero bracket, and $\ell_3$ is given for $1$-forms $x,y,z\in \mathfrak{g}^*$ by  by the formula 
\begin{equation}
\ell_3(x,y,z)=\iota_{\eta^{-1}}(x\wedge y\wedge z), 
\end{equation}
(where $\eta^{-1}$ is the inverse of the Cartan 3-form relative to the inner product) and extended to all forms by requiring the Leibniz rule in each argument.
\end{theorem}

This $L_\infty$ algebra
controls the deformation theory of $\Delta$ as a Lagrangian subalgebra
of $\mathfrak{d}$. The variety of Lagrangian subalgebras in $\mathfrak{d}$ was studied extensively in \cite{EvLu}, where 
in particular, the authors described the irreducible components of the variety and showed that they are all smooth. 
As we know from \cite{EvLu}, the space of Lagrangian subalgebras in $\mathfrak{d}$ is \textit{smooth} near $\Delta$. 
This suggests that the $L_\infty$ algebra given by Theorem \ref{Thm:cartan_vs_gauss_three_term_L_infty} might be formal. We show in the following section (Corollary 
\ref{Cor:formality_of_cartan_vs_gauss_three_term_L_infty}) that this is indeed the case when $\mathfrak{g}$ is a simple complex Lie algebra.

Let $G$ be a Lie group integrating $\mathfrak{g}$. 
Then, as we mentioned above, one has an isomorphism 
$\TT G \cong G\times (\mathfrak{g}\oplus \overline{\mathfrak{g}})$, 
preserving the pairing and the bracket. Thus, the Lagrangian subalgebras of $\mathfrak{d}$ correspond to 
Dirac structures on $G$.
For instance, the diagonal $\Delta \subset \mathfrak{d}$ corresponds to the so-called Cartan-Dirac structure $E_G$\cite{ABM}. 
The antidiagonal $\overline{\Delta}$ corresponds to an almost Dirac structure $\hat{F}_G$, transverse to $E_G$. 
Then, one can study the deformation theory of $E_G$ from the viewpoint of $\hat{F}_G$. 
The relevant $L_\infty$ algebra is obtained from the one constructed in Theorem \eqref{Thm:cartan_vs_gauss_three_term_L_infty}; 
i.e. we embed $\wedge \mathfrak{g}^*$ into the complex of differential forms on $G$ as right-invariant forms, 
and then extend both the differential and the triple bracket to all forms by requiring Leibniz rule in each entry.

\subsection{Lie bialgebras}
Let $\mathfrak{g}$ be a Lie bialgebra. 
Then the double $\mathfrak{g}\oplus \mathfrak{g}^*$ carries a unique Lie bracket that extends the brackets 
on $\mathfrak{g}$ and $\mathfrak{g}^*$ and is invariant under the natural symmetric pairing. 
The space $\mathfrak{g}\oplus \mathfrak{g}^*$ is an example of general Courant algebroid.
Our techniques work in this situation equally well as in the situation of the exact Courant algebroid $T\oplus T^*$.

Specifically, let $L$ and $M$ be a pair of two transverse Lagrangian Lie subalgebras inside $\mathfrak{g}\oplus \mathfrak{g}^*$. 
Then using the identification $M^*\cong L$, we can endow the exterior algebra $\wedge M^*[1]$ with the strcture of DGLA. 
The DGLA controls the deformation theory of the Lie bialgebra $M$ preserving its double. 
Using the techniques we describe in Section \ref{section:proof}, one can show that choice of another Lagrangian Lie subalgebra $L_\varepsilon$, 
transverse to $M$, leads to an $L_\infty$-isomorphic DGLA.
Moreover, this $L_\infty$-isomorphism is determined canonically by the relative position of $L$ and $L_\varepsilon$.

Now, assume that $(\mathfrak{g},[~,~],[~,~]_*)$ is a triangular Lie bialgebra, i.e. the Lie bracket on $[~,~]_*$ is given 
as the $1$-coboundary $\delta \varepsilon$ for an $r$-matrix $\varepsilon \in \wedge^2\mathfrak{g}$ satisfying the classical Yang-Baxter equation 
$[\varepsilon,\varepsilon]=0$, 
where $\delta$ is the Chevalley-Eilenberg differential 
$\wedge (\mathfrak{g}^*)\otimes\wedge^2 \mathfrak{g} \to \wedge^{\bullet+1} (\mathfrak{g}^*) \otimes \wedge^2 \mathfrak{g}$ 
for the $\mathfrak{g}$-module $\wedge^2 \mathfrak{g}$. Consider $L=graph(\varepsilon:\mathfrak{g}^*\to \mathfrak{g})$ inside the
double $\mathfrak{g}\oplus \mathfrak{g}^*$ of the Lie bialgebra  $(\mathfrak{g},[~,~],0)$ with zero cobracket. 
Then one can check that $L$ is a Lagrangian subalgebra, and $[x^*,y^*]_*=[\![ \varepsilon(x^*),\varepsilon(y^*) ]\!]$, $x^*,y^*\in \mathfrak{g}^*$, 
where $[\![~, ~]\!]$ is the bracket on the double $\mathfrak{g}\oplus \mathfrak{g}^*$. Thus, we are in a situation for which our techniques apply,
and we arrive at the following

\begin{theorem}
 Let $(\mathfrak{g},[~,~],[~,~]_*)$ be a triangular Lie bialgebra, given by an $r$-matrix $\varepsilon \in \wedge^2\mathfrak{g}$.
Then the DGLA $((\wedge \mathfrak{g}^*)[1], d, [~,~]_*)$ is $L_\infty$-formal. 
\end{theorem}

\begin{proof}
As in Corollary \ref{cor:independence_of_L_infty_from_the_choice_of_complement}, the $L_\infty$-automorphism sending the DGLA $((\wedge \mathfrak{g}^*)[1], d, [~,~]_*)$ to the abelian DGLA $((\wedge \mathfrak{g}^*)[1], d)$
is given by $e^{-R_\varepsilon}$. 
\end{proof}

Let us now discuss the case of quasi-triangular $r$-matrix.
Let $\mathfrak{g}$ be a simple complex Lie algebra, and let 
$\langle\cdot,\cdot\rangle \in (Sym^2\mathfrak{g}^*)^\mathfrak{g}$ be a non-degenerate symmetric invariant pairing (whish has to be a multiple of the Killing form).
Fix a decomposition $\mathfrak{g} = \mathfrak{n_-}\oplus \mathfrak{h} \oplus \mathfrak{n_+}$ into the negative nilpotent, Cartan, 
and the positive nilpotent subalgebras.
Let $\eta \in (\wedge^3 \mathfrak{g})^\mathfrak{g}$ be the corresponding Cartan $3$-form, i.e. $\eta(x,y,z) =\frac{1}{2} \langle[x,y],z\rangle$, $x,y,z\in \mathfrak{g}$.
The theorem of Belavin-Drinfeld \cite{BD} describes the solutions of the modified Yang-Baxter equation
\begin{equation}\label{CYBE}
 [r,r]= \eta^{-1},
\end{equation}
where $r\in \wedge^2 \mathfrak{g}$, and $\eta^{-1}$ is the inverse of $\eta$ relative to $\langle\cdot,\cdot\rangle$. 
Such solutions $r$ are called quasi-triangular $r$-matrices,
and are parametrized, up to a sign, by Belavin-Drinfeld triples (which are pairs of two subsets $\Gamma_1,\Gamma_2$ of the corresponding Dynkin diagram
and a a bijection between them satisfying certain nilpotency condition). The simplest solution $r_{st}$ to the classical Yang-Baxter equation is given by 
the skew-symmetrization of the canonical tensor $Id\in \mathfrak{n_-}^*\otimes \mathfrak{n_-}\equiv \mathfrak{n_+}\otimes \mathfrak{n_-}$.

Every solution of \eqref{CYBE} gives rise to a cobracket $\delta r$ on $\mathfrak{g}$, thus defining a Lie bialgebra strcture on $\mathfrak{g}$.
As in the case of triangular Lie bialgebras, one can ask whether the DGLA $((\wedge \mathfrak{g}^*)[1], d, [~,~]_*)$ is $L_\infty$-formal or not.
The answer is again affirmative, although this time the proof is not immediate and requires integration to the group $G$.

\begin{theorem} \label{formality_of_quasitriangular_Lia_bialgebra_DGLA} 
Let $\mathfrak{g}$ be a simple complex Lie algebra, and $(\mathfrak{g},[~,~],[~,~]_*,r)$ be a quasi-triangular Lie bialgebra structure on it.
Then the DGLA $((\wedge \mathfrak{g}^*)[1], d, [~,~]_*)$ is $L_\infty$-formal. 
\end{theorem}

\begin{proof}
Let $G$ be the simply connected Lie group integrating $\mathfrak{g}$.
By Drinfeld's theorem, the Lia bialgebra $\mathfrak{g}$ integrates to a Poisson-Lie group $(G,\pi)$.
The DGLA $((\wedge \mathfrak{g}^*)[1], d, [~,~]_*)$ embeds into the de Rham-Koszul DGLA $(\Omega_G[1], d, [~,~]_\pi)$
as the sub-DGLA of right invariant forms $(\Omega_G)^{R}[1]$.
Since the group $G$ is simple, the inclusion map $(\Omega_G)^{R}\hookrightarrow \Omega_G$ is quasi-isomorphism. 
To see the that, one needs to use the fact that $G$ deformation retracts onto 
its compact form and then apply the standard averaging technique there.

Now, we reduced the question of $L_\infty$-formality of the finite dimensional DGLA $((\wedge \mathfrak{g}^*)[1], d, [~,~]_*)$ to the $L_\infty$-formality of the infinite-dimensional
DGLA $(\Omega_G[1], d, [~,~]_\pi)$. However, we already know by Theorem~\eqref{Thm:Formality_of_Koszul_DGLA} that the latter is indeed formal.
\end{proof}

\begin{corollary}\label{Cor:formality_of_cartan_vs_gauss_three_term_L_infty}
 If $\mathfrak{g}$ is a simple complex Lie algebra, then the $L_\infty$ algebra constructed in the Theorem \eqref{Thm:cartan_vs_gauss_three_term_L_infty} is formal.
\end{corollary}
\begin{proof}
 Choose a quasi-triangular $r$-matrix, say $r_{st}$. Then we can view $r\in \wedge^2\mathfrak{g}$ as a skew map $\overline{\Delta}\to \Delta$, 
where $\overline{\Delta}\cong \mathfrak{g}^*$, $\Delta\cong \mathfrak{g}$ (we use the notations from Section \ref{Subsec:lie_groups}).
Then the graph of $L=graph\{r:\overline{\Delta}\to \Delta\}$ inside $\mathfrak{d}$ is a Lagrangian subalgebra, transverse to $\Delta$.
The splitting $(\Delta,\overline{\Delta})$ induces the $L_\infty$ algebra in Theorem \eqref{Thm:cartan_vs_gauss_three_term_L_infty}, 
while the splitting $(\Delta,L)$ induces the $L_\infty$ algebra $((\wedge \mathfrak{g}^*)[1], d, [~,~]_*)$. 
Therefore, as in Corollary \ref{cor:independence_of_L_infty_from_the_choice_of_complement}, these two $L_\infty$ algebras are $L_\infty$-isomorphic.
Theorem \ref{formality_of_quasitriangular_Lia_bialgebra_DGLA} ensures $L_\infty$-formality of the latter one, so we deduce this way $L_\infty$-formality of the former one.
\end{proof}

\subsection{Complex manifolds}\label{KSTH}
In this section we explain how our results may be used to study the deformation theory of a fixed deformation of a complex manifold $X$. In Theorem~\ref{thm:exthml}, we explain how a simple modification of the original Kodaira-Spencer deformation complex may be used to describe the deformations of a particular small deformation of $X$.

Let $I$ be the complex structure on the manifold $X$. The endomorphism $I$ determines and is determined by a pair of transverse complex Dirac structures 
\[
M=T_{1,0}\oplus T^*_{0,1},\qquad L=T_{0,1}\oplus T^*_{1,0}.
\]
The Kodaira-Spencer DGLA $(\Omega^{0,\bullet}(T_{1,0}), \delbar, [~,~])$ controlling deformations of complex structure is a sub-DGLA of the one controlling deformations of $M$, which is the ``extended deformation complex'' given by 
\[
(\Omega_M[1], d_M, [~,~]_L) = (\Omega^{0,\bullet}(\wedge^{\bullet}T_{1,0})[1], \delbar, [~,~]).
\]
While the results of Section~\ref{section:proof} concern deformations of Dirac structures, we may in this case apply them to the Kodaira-Spencer sub-DGLA.

We begin with a small deformation of the complex structure $I$, described by the Maurer-Cartan element $\phi\in\Omega^{0,1}(T_{1,0})$.  Viewing this element as a map $\phi:T_{0,1}\to T_{1,0}$, we obtain a description of the deformed Dirac structures $M^\phi, L^\phi$ as graphs of maps $\Phi:M\to L$, 
$\overline{\Phi}: L \to M$ respectively, given by the block matrix
$$
\begin{array}{llc}
\Phi = &
\begin{pmatrix}
\phi & & 0 \\
\\
0 & & - \phi^* 
\end{pmatrix}
:&
\begin{matrix}
T_{0,1} & &T_{1,0} \\
\oplus & \longrightarrow &  \oplus\\
T_{1,0}^* & & T_{0,1}^* 
\end{matrix}
\end{array}
$$
and its complex conjugate.  Therefore, the tensor $\Phi$ defines a Maurer-Cartan element in the larger DGLA, that is, viewed as an element $\Phi\in\Omega^2_M$ we have 
\begin{equation}\label{eq:MC_equation_for_Phi}
d_M \Phi + \frac{1}{2}[\Phi,\Phi]_L = 0.
\end{equation}

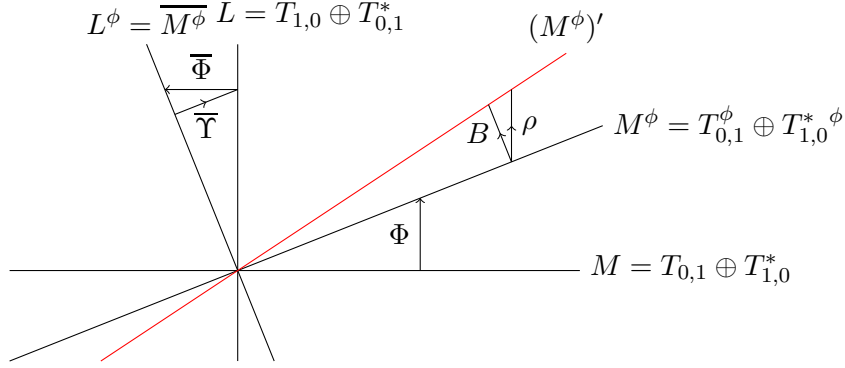
\begin{figure}[!ht]
\centering
\begin{tikzpicture}[scale=3,decoration={
    markings,
    mark=at position 0.5 with {\arrow{>}}}]

    \draw (1,0.6) coordinate (L1) -- (1,2) coordinate[label=above:{$~~~~~~~~~~~~~~~L=T_{1,0}\oplus T^*_{0,1}$}] (L2);
    \draw (0,1) coordinate (M1) -- (2.5,1) coordinate[label=right:{$M=T_{0,1}\oplus T^*_{1,0}$}] (M2) ;
    \coordinate (O) at (intersection of L1--L2 and M1--M2);
	\coordinate (Lphi1) at (1.16,0.6); 
	\coordinate[label=above:{$L^\phi=\overline{M^\phi}$}] (Lphi2) at ($(Lphi1)!3.5!(O)$); 
	\draw (Lphi1)--(Lphi2);
	\coordinate (Phibar1) at ($(O)+(0,0.8)$); 
	\coordinate (Phibar0) at ($(Phibar1)+(1,0)$);
	\coordinate (Phibar2) at (intersection of Lphi1--Lphi2 and Phibar1--Phibar0); 
	\draw[->] (Phibar1) to node [above] {$\overline{\Phi}$} (Phibar2);
	\coordinate (Mphi1) at (0, 0.6);
	\coordinate[label=right:{$M^\phi=T_{0,1}^\phi\oplus {T^*_{1,0}}^\phi$}] (Mphi2) at ($(Mphi1)!2.6!(O)$); 
	\draw (Mphi1)--(Mphi2);
	\coordinate (Psi3) at ($(Phibar1) +0.3*(Mphi1)-0.3*(Mphi2)$);
	\coordinate (Psi2) at (intersection of Phibar1--Psi3 and Lphi1--Lphi2);
	\draw[postaction={decorate}] (Psi2) to node[below] {$\overline{\Upsilon}$} (Phibar1);
	\coordinate (M_def1) at (0.4, 0.6); 
	\coordinate[label=above:{$(M^\phi)'$}] (M_def2) at ($(M_def1)!3.4!(O)$); 
	\draw[red] (M_def1)--(M_def2);
	\coordinate (rho2) at ($(M_def1)!3!(O)$); 
	\coordinate (pho0) at ($(rho2)-0.3*(L2)+0.3*(L1)$);
	\coordinate (m3) at (intersection of Mphi1--Mphi2 and rho2--pho0);
	\coordinate (g3) at ($(m3)-0.3*(Lphi2)+0.3*(Lphi1)$); 
	\coordinate (m4) at (intersection of M_def1--M_def2 and m3--g3);
	\draw[postaction={decorate}] (m3) to node [left] {$B$} (m4);
	\draw[postaction={decorate}] (m3)to node [right] {$\rho$} (rho2);
	
	\coordinate (Phi1) at ($(O)+(0.8,0)$); 
	\coordinate (Phi0) at ($(Phi1)+(0,1)$);
	\coordinate (Phi2) at (intersection of Mphi1--Mphi2 and Phi1--Phi0); 
	\draw[->] (Phi1) to node [left] {${\Phi}$} (Phi2);
\end{tikzpicture}
\caption{Description of a deformation $(M^\phi)'$ of a deformation $M^\phi$ of $M$. The Maurer-Cartan elements $B$ and $\rho$ are in $L_\infty$-isomorphic DGLAs.}
\end{figure}

To describe the deformations of the complex structure $I^\phi$, one would normally consider the Kodaira-Spencer DGLA of the deformed complex structure, 
that is the DGLA $(\Omega_\phi^{0,\bullet}(T^\phi_{1,0}), \delbar_\phi, [~,~]_\phi)$, where $T^\phi_{1,0} = (\Id + \ol\phi)T_{1,0}$ is the deformed holomorphic tangent bundle. As before, this is a sub-DGLA of the deformation complex $(\Omega_{M^\phi}[1], d_{M^\phi}, [~,~]_{L^\phi})$ associated to the pair of Dirac structures $(M^\phi,L^\phi)$. 

Instead of using the deformed Kodaira-Spencer DGLA, we may select an alternative complement to $M^\phi$ and employ the DGLA associated to the pair of Dirac structures $(M^\phi, L)$; by Corollary~\ref{cor:independence_of_L_infty_from_the_choice_of_complement}, the resulting DGLA describes the same deformation theory (i.e., deformations of $M^\phi$).  Furthermore, we may use the isomorphism $(\Id +\Phi):M\to M^\phi$ to transport this alternative DGLA structure from $\Omega^\bullet_{M^\phi}$ to $\Omega^\bullet_M$; this provides a new DGLA structure on $\Omega^\bullet_M$ controlling deformations of $I^\phi$.  In the remainder of this section, we describe the details of this procedure.

\begin{lemma}\label{lm: KS_deformed_differential}
 The vector bundle map $(\Id + \Phi):M\to M^\phi$ induces a DGLA isomorphism 
\[
(\Omega_{M^\phi}[1], d_{M^\phi}, [~,~]_L) \longrightarrow (\Omega_M[1], d_M + [\Phi,-]_L, [~,~]_L)
\]
taking the DGLA $(\Omega^{0,\bullet}_\phi(T^\phi_{1,0}), \delbar_\phi, [~,~])$ to the original Kodaira-Spencer complex,  equipped with deformed differential, 
namely to the DGLA 
$
(\Omega^{0,\bullet}(T_{1,0}), \delbar', [~,~]),
$
where 
\[
\delbar'=\del+[\phi,-].
\]
\end{lemma}

 \begin{proof}
 Note that the DGLA  $(\Omega_{M^\phi}[1], d_{M^\phi}, [~,~]_L)$  
 can be obtained by the derived bracket construction \eqref{eq:quantized_derived_brackets} from the $BV_\infty$ torsor
 $(\Omega_{M^\phi}\cong \mathcal{X}_L , \Omega, d)$. One checks that the map $e^{\Phi}:\Omega\to \Omega$ gives 
 an isomorphism of $BV_\infty$ torsors $(\Omega_{M^\phi}\cong \mathcal{X}_L , \Omega, d) \to 
 (\Omega_{M}\cong\mathcal{X}_L , \Omega, e^{\Phi}de^{-\Phi})$. 
We remark that the map $\wedge(\Id + \Phi)^*$ can alternatively be described as the composition of the canonical identifications 
$\Omega_{M^\phi}\cong \mathcal{X}_L \cong \Omega_M$. 
Let $d=d_{-1}+d_{+1}$ be the degree decomposition of $d$ with respect to the spinor grading \eqref{eq:spinor_grading}.
Note that $\Phi$ has degree $+2$, so the degree decomposition of $e^{\Phi}de^{-\Phi}$ is $d_{-1} + ([\Phi,d_{-1}] + d_{+1})$ 
(the term $\frac{1}{2}[\Phi,[\Phi,d_{-1}]]+[\Phi,d_{+1}]$ vanishes because $\Phi$ satisfies the MC equation \eqref{eq:MC_equation_for_Phi}). 
We see that the degree $-1$ component didn't change, which means it induces the same bracket on $\mathcal{X}_L$. 
The degree $+1$ component acquired the extra term $[\Phi,d_{-1}]$. This leads to the deformed differential $d_M + [\Phi,-]_L$.

For the final statement, one checks that the map $\wedge(\Id + \Phi)^*$ preserves the double grading 
$\Omega_M = \Omega^{0,\bullet}(\wedge^{\bullet}T_{1,0})$.
 \end{proof}


\begin{theorem}\label{thm:exthml}
 Let $(X,I^\phi)$ be a small deformation of the complex manifold $(X,I)$.  
Then the Kodaira-Spencer DGLA 
$(\Omega_\phi^{0,\bullet}(T^\phi_{1,0}), \delbar_\phi, [~,~]_\phi)$ which controls deformations of the complex structure 
$I^\phi$ is $L_\infty$-isomorphic to the DGLA $(\Omega^{0,\bullet}(T_{1,0}), \delbar', [~,~])$. 
\end{theorem}

\begin{proof}
Let us fix the map $\overline{\Upsilon}:L^\phi\to M^\phi$ such that
$L=graph\{\overline{\Upsilon}:L^\phi\to M^\phi\}$. 
In practice, $\overline{\Upsilon}$ can be computed as $\overline{\Upsilon}=(1+\overline{\Psi})^{-1}-1$, where $\overline{\Psi}:L\to M^\phi$ is given by
${\Psi}=(1+\overline{\Phi}){\Phi} (1 - \overline{\Phi} \Phi )^{-1}$.
Corollary \ref{cor:independence_of_L_infty_from_the_choice_of_complement} implies that $e^{R_{\overline{\Upsilon}}}$ gives an $L_\infty$-isomorphism between the 
DGLAs $(\Omega_{M^\phi}[1], d_{M^\phi} , [~,~]_{L^\phi})$ and $(\Omega_{M^\phi}[1], d_{M^\phi}, [~,~]_L)$. 
The latter one by Lemma~\ref{lm: KS_deformed_differential} is isomorphic as DGLA to $(\Omega_M[1], d_M + [\Phi,-]_L, [~,~]_L)$. 
It remains to note that since $\overline{\Upsilon}\in\Omega^{1,0}_\phi(T_{01}^\phi)$, the all the components of the $L_\infty$ map 
$e^{R_{\overline{\Upsilon}}}$ preserve 
$\Omega^{0,\bullet}_\phi(T_{10}^\phi)$.
\end{proof}

The theorem above establishes that even though the deformation complex \linebreak $(\Omega_\phi^{0,\bullet}(T^\phi_{1,0}), \delbar_\phi, [~,~]_\phi)$ of $I^\phi$ may not be isomorphic as a DGLA to the modified deformation complex $(\Omega^{0,\bullet}(T_{1,0}), \delbar', [~,~])$ of $I$, they are $L_\infty$--isomorphic.  Furthermore,  Lemma~\ref{lm: KS_deformed_differential} provides an isomorphism $\wedge(\Id + \Phi)^*$ of cochain complexes between these two DGLAs. Therefore, we may use this isomorphism to transport the Lie bracket $[~,~]_\phi$ to a Lie bracket $[~,~]'$ on the original Dolbeault complex, endowing it with two separate Lie algebra structures:
\begin{equation}
(\Omega^{0,\bullet}(T_{1,0}), \delbar', [~,~], [~,~]').
\end{equation}
Reiterating, the bracket $[~,~]$ is the Schouten-Nijenhuis bracket of the original complex manifold, while the bracket $[~,~]'$ is the corresponding bracket of the deformation $I^\phi$, transported by $\wedge(\Id + \Phi)^*$.  Each of these defines a DGLA structure on the same underlying differential complex. 

Of course, by construction and using Theorem~\ref{thm:exthml}, this pair of DGLA structures is $L_\infty$ isomorphic, via the conjugated equivalence 
\[
e^{R_{E}}=(1+\Phi)^*\circ e^{R_{\overline{\Upsilon}}} \circ ((1+\Phi)^*)^{-1},
\]
which may be solved for $E$, yielding 
$$
\begin{array}{llc}
E = &
\begin{pmatrix}
\varepsilon & & 0 \\
\\
0 & & - \varepsilon^* 
\end{pmatrix}
:&
\begin{matrix}
T_{1,0} & &T_{0,1} \\
\oplus & \longrightarrow &  \oplus\\
T_{0,1}^* & & T_{1,0}^*
\end{matrix}
\end{array}
$$
where  $\varepsilon = - \overline{\phi} (1-\phi \overline{\phi})^{-1} \in \Omega^{1,0}(T_{0,1})$.

%
%

The significance of this observation is that it will allow us to explicitly describe the bracket $[~,~]'$,  usually inaccessible since it depends upon the deformed holomorphic structure, in terms of the original deformation complex of $I$.
To do this, we describe the family of maps $f_1,f_2,...$ giving the $L_\infty$ morphism 
\begin{equation}\label{eq:l_inf_map_btw_dglas}
e^{R_{E}}:(\Omega^{0,\bullet}(T_{10}), \delbar', [~,~]')
\longrightarrow
(\Omega^{0,\bullet}(T_{10}), \delbar', [~,~]).
\end{equation}
The first map $f_1$ is the identity, and the rest are determined by the following bilinear operation on elements $\alpha_1,\alpha_2 \in \Omega^{0,\bullet}(T_{10})$:
\[
R_{\varepsilon}(\alpha_1,\alpha_2) = i_{\varepsilon}(\alpha_1 \wedge \alpha_2)  - (i_{\varepsilon} \alpha_1) \wedge \alpha_2 - \alpha_1 \wedge i_{\varepsilon}\alpha_2.
\]
This operation extends to a multilinear operator with $n$ inputs and $n-1$ outputs:
\[
R_{\varepsilon}(\alpha_1,\ldots,\alpha_n) = \sum_{i<j}\epsilon(\sigma) R_{\varepsilon}(\alpha_i, \alpha_j) \otimes \alpha_1 \otimes \cdots \otimes \widehat{\alpha}_i \otimes \cdots \otimes \widehat{\alpha}_j \otimes \cdots \otimes \alpha_n.
\]
The $k^\text{th}$ map $f_k$ in the $L_\infty$ morphism is then given by the $(k-1)$-fold composition 

\[
f_k(\alpha_1,\ldots, \alpha_k) = \frac{1}{(k-1)!}R_{\varepsilon}^{k-1}(\alpha_1,\ldots, \alpha_k).
\]
For example, we have $f_2(\alpha_1,\alpha_2)=R_{\varepsilon}(\alpha_1,\alpha_2)$, and 
\[
\begin{array}{ll}
f_3(\alpha_1,\alpha_2,\alpha_3) = \frac{1}{2!}\Big(R_{\varepsilon}(R_{\varepsilon}(\alpha_1,\alpha_2),\alpha_3)+ 
(-1)^{(|\alpha_2|+1)(|\alpha_3|+1)}R_{\varepsilon}(R_{\varepsilon}(\alpha_1,\alpha_3),\alpha_2)+\\
(-1)^{(|\alpha_1|+1)(|\alpha_2|+|\alpha_3|)}R_{\varepsilon}(R_{\varepsilon}(\alpha_2,\alpha_3),\alpha_1)\Big).
\end{array}
\]

\begin{corollary}\label{cor:L_infty_map_second_equation}
Let $D=[\delbar', \iota_\varepsilon]: \Omega^{0,\bullet}(T_{\bullet,0})\to \Omega^{0,\bullet}(T_{\bullet-1,0})$. 
Then for any $u,v\in \Omega^{0,\bullet}(T_{10})$ one has
\begin{equation}\label{eq:L_infty_map_second_equation}
 [u,v]' = [u,v] + (-1)^{|u|+1}(D(u\wedge v) - D(u) \wedge v) - u \wedge D(v).
\end{equation}
\end{corollary}

\begin{proof}
The fact that the family of maps $\{f_1=Id, f_2, f_3, ...\}$ form an $L_\infty$ morphism, implies, using \eqref{eq:L_inf_relation_2} and \eqref{eq:symm_to_skew}, that 
\begin{equation*}
(-1)^{|u|+1}([u,v]' -[u,v]) = \delbar'(f_2(u,v)) - f_2(\delbar'(u),v) - (-1)^{|u|+1}f_2(u,\delbar'(v)),
\end{equation*}
for $u,v\in \Omega^{0,\bullet}(T_{10})$. Expanding the right hand side gives \eqref{eq:L_infty_map_second_equation}.
\end{proof}

\begin{remark}
$D$ is a degree $-1$, $2^{nd}$ order differential operator on the algebra $\Omega^{0,\bullet}(T_{\bullet,0})$ which squares to zero, and so it defines a BV operator.  Instead of generating the bracket, as it would do in a BV algebra, it generates the difference between the two brackets in question.  In other words, we have obtained a Tian-Todorov type formula relating the Schouten-Nijenhuis bracket of the original complex structure $I$ to that of the deformed structure $I^\phi$.
\end{remark}

Passing to cohomology, we obtain an easy corollary of the above result, implying for example that the quadratic obstruction maps for the two brackets coincide:

\begin{corollary}\label{cor:brackets_on_cohomology}
 The brackets $[-,-]$ and $[-,-]'$ induce the same Lie algebra on cohomology
$$
H^k_{\delbar'}(\Omega^{0,\bullet}(T_{10}))\times H^l_{\delbar'}(\Omega^{0,\bullet}(T_{10}))
 \longrightarrow H^{k+l}_{\delbar'}(\Omega^{0,\bullet}(T_{10}))
$$
\end{corollary}

\begin{proof}
 If $u \in \Omega^{0,k}(T_{10}), v \in\Omega^{0,l}(T_{10})$ are $\delbar'$-closed, then one can replace $D$ with $\delbar' \iota_\varepsilon$ in
the right hand side of \eqref{eq:L_infty_map_second_equation}. Then it simplifies to $(-1)^{|u|+1}\delbar'(R_{\varepsilon}(u,v))$.
\end{proof}

Finally, applying Theorem \ref{mccomp} to the $L_\infty$ morphism 
\eqref{eq:l_inf_map_btw_dglas}, we obtain an explicit map taking 
Maurer-Cartan elements $\rho\in (\Omega^{0,\bullet}(T_{1,0}), \delbar', [-,-])$ to Maurer-Cartan elements for the DGLA $(\Omega^{0,\bullet}(T_{1,0}), \delbar', [-,-]')$.    For $\rho$ such that $(1+\rho\eps)$ is invertible, we obtain the Maurer-Cartan element 
\begin{equation}
B = (1 + \rho\eps)^{-1}\rho  = \rho -\rho\eps\rho + \rho\eps\rho\eps\rho - \cdots.
\end{equation}
%
%
%
%
%


\begin{thebibliography}{100}

\bibitem{AKSM} A. Alekseev, Y. Kosmann-Schwarzbach, E. Meinrenken, 
\textit{Quasi-Poisson manifolds.}
Canad. J. Math. 54 (2002), no. 1, 3--29.

\bibitem{ABM}
A. Alekseev, H. Bursztyn, E. Meinrenken,
\textit{Pure spinors on Lie groups.}
Ast\'{e}risque No. 327 (2009), 131--199 (2010). 

\bibitem{AMM}
A. Alekseev, A. Malkin and E. Meinrenken, 
\textit{Lie group valued moment maps.} J. Differential Geom. 48(1998), 445--495. 

\bibitem{BaVo} D. Bashkirov, A. Voronov,
\textit{ The BV formalism for $L_\infty$-algebras.}
\href{http://arxiv.org/abs/1410.6432v2}{arxiv.org/abs/1410.6432v2}, to appear in Journal of Homotopy and Related Structures.

\bibitem{BD} A.A. Belavin, V.G. Drinfel'd, 
\textsl{Triangle equations and simple Lie algebras.}
Translated from the Russian. Soviet Sci. Rev. Sect. C Math. Phys. Rev., 4, 
Mathematical physics reviews, Vol. 4, 93--165, Harwood Academic Publ., Chur, 1984. 

%
\bibitem{BDA} K. Bering, P. Damgaard, J. Alfaro,
\textit{Algebra of higher antibrackets.}
Nuclear Phys. B 478 (1996), no. 1-2, 459--503.

%
\bibitem{CaSch} A. Cattaneo, F. Sch\"{a}tz, 
\textit{Equivalences of higher derived brackets.}
J. Pure Appl. Algebra 212 (2008), no. 11, 2450--2460. 

%
\bibitem{Co} T. Courant,
\textit{Dirac Manifolds}
Transactions of the AMS, 319 (1990), no. 2, 631--661.

\bibitem{EvLu} 
S. Evens, J.-H. Lu,
{On the variety of Lagrangian subalgebras. II.}
Ann. Sci. \`{E}cole Norm. Sup. (4) 39 (2006), no. 2, 347--379. 

\bibitem{FM12}
D. Fiorenza, M. Manetti, \textit{Formality of Koszul brackets and deformations of holomorphic Poisson manifolds.}
Homology Homotopy Appl. 14 (2012), no. 2, 63--75.

\bibitem{FlZa} F. Sch\"atz, M. Zambon, \textit{Deformations of 
pre-symplectic structures and the Koszul $L_\infty$-algebra}, Preprint (2017).

\bibitem{FrZa} Y. Fr\'{e}gier, M. Zambon, \textit{
Simultaneous deformations and Poisson geometry.}
Compos. Math. 151 (2015), no. 9, 1763--1790. 

\bibitem{GV95} M. Gerstenhaber, A. Voronov,
\textsl{Homotopy G-algebras and moduli space operad.}
Internat. Math. Res. Notices 1995, no. 3, 141--153.

\bibitem{GoMi90} W. Goldman, J. Millson,
\textsl{The homotopy invariance of the Kuranishi space.}
Illinois J. Math. 34 (1990), no. 2, 337--367. 

\bibitem{Gua11}M.Gualtieri, 
\textit{Generalized complex geometry.}
Ann. of Math. (2) 174 (2011), no. 1, 75--123. 

\bibitem{Hi12} N. Hitchin,
\textsl{Deformations of holomorphic Poisson manifolds.}
Mosc. Math. J. 12 (2012), no. 3, 567--591.


\bibitem{KeWa} Frank Keller, Stefan Waldmann,
\textsl{Formal deformations of Dirac structures.}
J. Geom. Phys. 57 (2007), no. 3, 1015--1036. 


\bibitem{KS04}
Y. Kosmann-Schwarzbach,
\textit{Derived brackets.} 
Lett. Math. Phys. 69 (2004), 61--87. 

\bibitem{Kra00} O. Kravchenko,
\textit{Deformations of Batalin-Vilkovisky algebras.} Poisson geometry (Warsaw, 1998), 131--139,
Banach Center Publ., 51, Polish Acad. Sci., Warsaw, 2000. 

\bibitem{LiBS}
D. Li-Bland, P. \v{S}evera, 
\textit{Moduli spaces for quilted surfaces and Poisson structures.}
Doc. Math. 20 (2015), 1071--1135. 

\bibitem{LWX}
Z.-J. Liu, A. Weinstein, P. Xu,
\textit{Manin triples for Lie bialgebroids.}
J. Differential Geom. 45 (1997), no. 3, 547-574. 

\bibitem{Roy} D. Roytenberg, \textit{On the structure of graded symplectic supermanifolds and Courant algebroids.} 
Quantization, Poisson brackets and beyond (Manchester, 2001), 169--185.

\bibitem{Roy1} D. Roytenberg,
\textit{Quasi-Lie bialgebroids and twisted Poisson manifolds.}
{Lett. Math. Phys.} 61 (2002), no. 2, 123--137. 

\bibitem{Sevwein} \v{S}evera, P. and Weinstein, A.,
\textit{Poisson Geometry with a 3-Form Background.}
{Progress of Theoretical Physics Supplement}, 144 (2002), 145--154.

\bibitem{ShTa08}
G. Sharygin, D. Talalaev,
\textit{On the Lie-formality of Poisson manifolds.}
J. K-Theory 2 (2008), no. 2, Special issue in memory of Yurii Petrovich Solovyev. Part 1, 361--384. 

\bibitem{Tam98} D. E. Tamarkin,
\textsl{Another proof  of M. Kontsevich formality theorem.} 
Preprint, The Pennsylvania State University, March 1998,
\href{https://arxiv.org/abs/math/9803025}{arXiv:math/9803025}

\bibitem{Ti87} G. Tian, 
\textit{Smoothness of the universal deformation space of compact Calabi-Yau manifolds and its Petersson-Weil metric.}
Mathematical aspects of string theory (San Diego, Calif., 1986), 629--646,
Adv. Ser. Math. Phys., 1, World Sci. Publishing, Singapore, 1987. 

\bibitem{Tod89} A. Todorov,
\textsl{The Weil-Petersson geometry of the moduli space of SU(n$\ge$3) (Calabi-Yau) manifolds. I.}
Comm. Math. Phys. 126 (1989), no. 2, 325--346. 

\bibitem{Vor} Th. Voronov, 
\textit{Higher derived brackets and homotopy algebras.}
J. Pure Appl. Algebra 202 (2005), no. 1-3, 133--153. 

\end{thebibliography}
\end{document}